\documentclass[12pt]{article}
 
 \usepackage{amsmath, latexsym, amsfonts, amssymb, amsthm, amscd, enumerate, comment, stackrel, mathrsfs}
 \usepackage[usenames,dvipsnames,svgnames,table]{xcolor}
 \usepackage[left=3cm, right=2.5cm, top=2.5cm, bottom=2.5cm]{geometry}

 \newtheorem{theorem}{Theorem}[section]
 \newtheorem{proposition}[theorem]{Proposition}
 \newtheorem{lemma}[theorem]{Lemma}
 
 \theoremstyle{definition}

 \title{Speed of extinction for CSBP in the  subcritical L\'evy environment:  strongly and intermediate cases}
 \author{Natalia Cardona-Tob\'on and Juan Carlos Pardo}

\begin{document}
 	
 	\begin{center}
 		{\Large \bf Speed of extinction for continuous state branching processes in a weakly subcritical L\'evy environment}\\[5mm]
 		
 		\vspace{0.7cm}
 		\textsc{Natalia Cardona-Tob\'on\footnote{Institute for Mathematical Stochastics, University of G\"ottingen.  Goldschmidtstrasse 7 C.P. 37077, G\"ottingen, Germany,  {\tt natalia.cardonatobon@uni-goettingen.de}}}  and \textsc{Juan Carlos Pardo\footnote{Centro de Investigaci\'on en Matem\'aticas. Calle Jalisco s/n. C.P. 36240, Guanajuato, M\'exico, {\tt jcpardo@cimat.mx}}}
 		
 	%	\vspace{0.5cm}
% 		Version of \today
 	\end{center}
 	
 	\vspace{0.3cm}

 	\begin{abstract}
 		\noindent 
 		In this manuscript, we continue with the systematic study of   the  speed  of extinction of continuous state branching processes in L\'evy environments under more general branching mechanisms. Here, we deal with  the weakly subcritical regime under the assumption  that the branching mechanism is regularly varying.
 We extend recent results of  Li and Xu \cite{li2018asymptotic} and Palau et al. \cite{palau2016asymptotic}, where it is assumed that the branching mechanism is stable  and complement the recent articles of Bansaye et al. \cite{bansaye2021extinction} and by the authors in \cite{CTP21}, where the critical   and  the strongly and intermediate subcritical cases were treated, respectively. Our methodology combines a path analysis of the branching process together with its L\'evy environment,   fluctuation theory for L\'evy processes and  the asymptotic behaviour of exponential functionals of L\'evy processes. Our approach is inspired by Afanasyev et al.\cite{afanasyev2012}, where the  discrete analogue was obtained, and by  \cite{bansaye2021extinction} and \cite{CTP21}.

 	\end{abstract}

 	\section{Introduction and main results}
 	%Galton-Watson processes in random environments (or GWREs for short) were introduced in the late sixties by Smith and Wilkinson \cite{smith1969} and  since then they have been deeply studied. More recently, the continuous analogue, in time and space, was also considered in different settings, see for instance \cite{bansaye2021extinction, bansaye2015scaling, boeinghoff2011branching, he2018continuous, palau2018branching} and references therein. The latter family of processes is now known as continuous state branching processes in random environments (or CBPRE for short).  Roughly speaking, a process in such class  is a $[0,\infty]$-valued Markov process, where $0$ and  $\infty$ are absorbing states, enjoying the quenched branching property; that is  conditionally on the environment,  the process started from $x+y$ is distributed as the  sum of two independent copies of the same process but issued from $x$ and $y$, respectively.  		 		
 	
 	In this manuscript we are interested in continuous state branching processes in random environments, in particular  when the   environment is driven by  a L\'evy process. This family of processes is known as \emph{continuous state branching processes in L\'evy environment} (or CBLEs for short) and they have been constructed independently   by He et al.  \cite{he2018continuous} and  Palau and Pardo \cite{palau2018branching},  as the unique non-negative strong solution of a stochastic differential equation whose linear term is driven by a L\'evy process.  
	
	The classification of  the asymptotic behaviour of rare events of CBLEs, such as the survival  probability, depends on the long-term behaviour of the environment.  In other words an auxiliary L\'evy process, which is associated to the environment, leads to the usual classification for the long-term behaviour of branching processes. To be more precise, the CBLE is called \textit{supercritical}, \textit{critical} or \textit{subcritical} accordingly as the auxilliary  L\'evy process drifts to $\infty$, oscillate or drifts to $-\infty$.  Furthermore, in the subcritical regime another phase transition arises which depends on whether the L\'evy process drifts to $\infty$, oscillate or drifts to $-\infty$ under a suitable exponential change of measure. These regimes are known in the literature as \textit{strongly}, \textit{intermediate} and \textit{weakly subcritical} regimes, respectively.

 	The study of the long-term behaviour of CBLEs  has attracted considerable attention in the last decade, see for instance  Bansaye et al. \cite{bansaye2013extinction}, B\"oinghoff and Hutzenthaler  \cite{boeinghoff2011branching}, He et al.  \cite{he2018continuous}, Li and Xu \cite{li2018asymptotic}, Palau and Pardo \cite{palau2017continuous, palau2018branching}, Palau et al. \cite{palau2016asymptotic} and Xu \cite{xu2021}.  All the aforementioned studies deal with the case when the branching mechanism is associated to a stable jump structure or a Brownian component on the branching term.  For simplicity on exposition we will call such branching mechanisms as {\it stable}. Bansaye et al. \cite{bansaye2013extinction} determined the long-term behaviour for  stable CBLEs when the random environment is driven by a L\'evy process with bounded variation paths. Palau and Pardo \cite{palau2017continuous}  studied the case when the random environment is driven by a Brownian motion with drift. Afterwards, Li and Xu \cite{li2018asymptotic} and Palau et al.  \cite{palau2018branching}, independently,  extended this result to the case when the environment is driven by a general L\'evy process.  More recently, Xu \cite{xu2021} provided an exact description for the speed of the extinction probability for CBLEs  with stable branching mechanism and where the L\'evy environment is heavy-tailed.  It is important to note that all these   manuscripts exploited the explicit knowledge of the survival probability which is given in terms of  exponential functionals of L\'evy processes.

	Much less is known about the long-term behaviour of CBLEs when the associated branching mechanism is  more general. Up to our knowledge, the only studies in this direction are  Bansaye et al.  \cite{bansaye2021extinction} and Cardona-Tob\'on and Pardo \cite{CTP21}, where  the speed of extinction for more general branching mechanisms  is studied. More precisely, Bansaye et al. \cite{bansaye2021extinction} focus on the critical case (oscillating L\'evy environments satisfying the so-called Spitzer's condition at $\infty$) and relax the assumption that the branching mechanism is  stable.  Shortly afterwards,  Cardona-Tob\'on and Pardo \cite{CTP21} studied the speed of extinction of CBLEs in the  strongly and intermediate subcritical regimes. Their methodology combines a path analysis of the branching process together with its L\'evy environment, fluctuation theory for L\'evy processes and the asymptotic behaviour of exponential functionals of L\'evy processes. 

		In this manuscript we continue with such systematic  study on the asymptotic behaviour of the survival probability for the CBLE under more general branching mechanisms   but now in the weakly subcritical regime. It is important to note that extending such asymptotic behaviour to more general branching mechanism is not as easy as we might think since we required to control a functional of the associated L\'evy process to the environment which is somehow quite involved. Moreover, contrary to the discrete case, the state 0 can be polar
and the process might become very close to 0 but never reach this point. To focus on the absorption event, we use Grey's condition which guarantees that 0 is accessible.

Our main contribution  is to provide its precise asymptotic behaviour  under some  assumptions on the auxiliary L\'evy process  and the branching mechanism.  		
		In particular, we obtain that the speed of the survival probability  decays exponentially with a polynomial factor of order ${3/2}$ (up to a multiplicative constant which is computed explicitly and depends on the limiting behaviour of the survival probability  given favorable environments). In particular, for the stable case we recover the results of \cite{li2018asymptotic}  where the limiting constant is given in terms of the exponential functional of the L\'evy process. In order to deduce such asymptotic behaviour, we combine  the approach developed in  \cite{{afanasyev2012}}, for the discrete time setting, with  fluctuation theory of L\'evy processes and a similar strategy developed by Bansaye et al. in \cite{bansaye2021extinction}. A key point in our arguments is to rewrite the probability of  survival under a suitable change of measure which is associated to an exponential martingale of the L\'evy environment. In order to do so, the existence of some exponential moments for the L\'evy environment is required. 
	Under this exponential change of measure  the L\'evy environment now oscillates
	and we can apply a similar strategy developed by Bansaye et al. in \cite{bansaye2021extinction} to study the extinction rate for CBLEs  in the critical regime. More precisely, under this new measure, we split the event of survival in two parts, that is when the running infimum is either negative or positive and then we show that only paths of the L\'evy process with a positive running infimum give substantial contribution to the speed of survival. In this regime, we assume that the  branching mechanism is regularly varying and a   lower bound for the branching mechanism which allow us to control the event of survival under favorable environments and unfavourable environments, respectively.  Our  results  complements those in \cite{bansaye2021extinction, CTP21}.

 	\subsection{Main results}\label{sec_defandprop}
 	Let $(\Omega^{(b)}, \mathcal{F}^{(b)}, (\mathcal{F}^{(b)}_t)_{t\geq 0}, \mathbb{P}^{(b)})$ be a filtered probability space satisfying the usual hypothesis on which we may construct the demographic (branching) term of the model that we are interested in. We suppose that $(B_t^{(b)}, t\geq 0)$ is a $(\mathcal{F}_t^{(b)})_{t\geq 0}$-adapted standard Brownian motion and  $N^{(b)}(\mathrm{d} s , \mathrm{d} z, \mathrm{d} u)$ is a  $(\mathcal{F}_t^{(b)})_{t\geq 0}$-adapted Poisson random measure on $\mathbb{R}^3_+$  with intensity $\mathrm{d} s \mu(\mathrm{d} z)\mathrm{d} u$ where  $\mu$ satisfies	\begin{equation}\label{eq_finitemean}
 		\int_{(0,\infty)}(z\wedge z^2)\mu(\mathrm{d}z)<\infty.
 	\end{equation} 
	We denote by $\widetilde{N}^{(b)}(\mathrm{d} s , \mathrm{d} z, \mathrm{d} u)$ for the compensated version of $N^{(b)}(\mathrm{d} s , \mathrm{d} z, \mathrm{d} u)$. Further, we also introduce  the  so-called branching mechanism $\psi$, a convex function with the following L\'evy-Khintchine representation
 \begin{equation}
 		\psi(\lambda) =\psi'(0+) \lambda  + \varrho^2 \lambda^2 + \int_{(0,\infty)} \big(e^{-\lambda x} - 1 + \lambda x \big) \mu(\mathrm{d} x), \qquad \lambda \geq 0,
 	\end{equation}
where $\varrho\ge 0$. Observe that the term $\psi'(0+)$ is well defined (finite) since condition \eqref{eq_finitemean} holds. Moreover, the function $\psi$ describes the stochastic dynamics of the population.
	
	On the other hand, for the environmental term, we consider another filtered probability space $(\Omega^{(e)}, \mathcal{F}^{(e)},(\mathcal{F}^{(e)}_t)_{t\geq 0}, \mathbb{P}^{(e)})$   satisfying the usual hypotheses. Let us consider $\sigma \geq 0$ and $\alpha$  real constants;  and $\pi$  a measure concentrated on $\mathbb{R}\setminus\{0\}$ such that $$\int_{\mathbb{R}} (1\land z^2)\pi(\mathrm{d} z)<\infty.$$ Suppose that \ $(B_t^{(e)}, t\geq 0)$ \ is a $(\mathcal{F}_t^{(e)})_{t\geq0}$ - adapted standard Brownian motion, $N^{(e)}(\mathrm{d} s, \mathrm{d} z)$ is a $(\mathcal{F}_t^{(e)})_{t\geq 0}$ - Poisson random measure on $\mathbb{R}_+ \times \mathbb{R}$ with intensity $\mathrm{d} s \pi(\mathrm{d} z)$, and $\widetilde{N}^{(e)}(\mathrm{d} s, \mathrm{d} z)$ its compensated version. We denote by $S=(S_t, t\geq 0)$ a L\'evy process, that is  a process with  stationary and independent increments and  c\`adl\`ag paths, with  the following L\'evy-It\^o decomposition
 	\begin{equation*}\label{eq_ambLevy}
 		S_t = \alpha t + \sigma B_t^{(e)} + \int_{0}^{t} \int_{(-1,1)} (e^z - 1) \widetilde{N}^{(e)}(\mathrm{d} s, \mathrm{d} z) + \int_{0}^{t} \int_{(-1,1)^c} (e^z - 1) N^{(e)}(\mathrm{d} s, \mathrm{d} z).
 	\end{equation*}
 	Note that   $S$ is a L\'evy process with no jumps smaller  than -1.

 	In our setting, we are considering independent processes for the demographic and  environmental terms. More precisely, we work now on the space $(\Omega, \mathcal{F}, (\mathcal{F}_t)_{t\geq 0}, \mathbb{P})$ the direct product of the two probability spaces defined above, that is to say, $\Omega := \Omega^{(e)} \times \Omega^{(b)}, \mathcal{F}:= \mathcal{F}^{(e)}\otimes \mathcal{F}^{(b)},  \mathcal{F}_t:=  \mathcal{F}^{(e)}_t \otimes  \mathcal{F}^{(b)}_t$ for $t\geq0$, $ \mathbb{P}:=\mathbb{P}^{(e)} \otimes \mathbb{P}^{(b)} $. Therefore $(Z_t, t\geq 0)$,
   the \textit{continuous-state branching process  in the L\'evy environment $(S_t, t\geq 0)$} is defined on  $(\Omega, \mathcal{F}, (\mathcal{F}_t)_{t\geq 0}, \mathbb{P})$ as the unique non-negative strong solution of the following stochastic differential equation 
   %(SDE for short)
	\begin{equation}\label{CBILRE}
	\begin{split}
 		Z_t = &Z_0 - \psi'(0+) \int_{0}^{t} Z_s \mathrm{d} s + \int_{0}^{t} \sqrt{2\varrho^2 Z_s} \mathrm{d} B_s^{(b)}    \\   &\hspace{4cm} +  \int_{0}^{t} \int_{(0,\infty)} \int_{0}^{Z_{s-}}z \widetilde{N}^{(b)} (\mathrm{d} s, \mathrm{d} z, \mathrm{d} u)+ \int_{0}^{t} Z_{s-} \mathrm{d} S_s.
		\end{split}
 	\end{equation}
 	According to  Theorem 3.1 in He et al. \cite {he2018continuous} or Theorem 1 in Palau and Pardo \cite{palau2018branching}, the equation has a unique positive strong solution which is not explosive. %The explosive case was also treated in \cite{palau2018branching} where it is shown that a modified version of the SDE \eqref{CBILRE} has a unique strong solution up to explosion and, by convention, it is identically equal to $\infty$ after the explosion time.
 An important property satisfied by $Z$ is that, given the environment, it inherits the branching property of the underlying continuous state branching process. We denote by $\mathbb{P}_z$, for its law starting from $z\ge 0$.
 	
 	The analysis of the process $Z$ is deeply related to the behaviour and fluctuations of the L\'evy process $\xi=(\xi_t, t\ge 0)$, defined as follows
 	\begin{equation}\label{eq_envir2}
 		\xi_t = \overline{\alpha} t + \sigma B_t^{(e)} + \int_{0}^{t} \int_{(-1,1)} z \widetilde{N}^{(e)}(\mathrm{d} s, \mathrm{d} z) + \int_{0}^{t} \int_{(-1,1)^c}z N^{(e)}(\mathrm{d} s, \mathrm{d} z),
 	\end{equation}
 	where
 	\begin{equation*}
 		\overline{\alpha} := \alpha -\psi'(0+)-\frac{\sigma^2}{2} - \int_{(-1,1)} (e^z -1 -z) \pi(\mathrm{d} z).
 	\end{equation*}
 	Note that, both processes $S$ and $\xi$ generate the same filtration. %Actually, the process $\xi$ is obtained from $S$, changing only the drift and jump sizes. 
	In addition, we see that the drift term $\overline{\alpha}$ provides the interaction between the demographic and environmental parameters. We denote by $\mathbb{P}^{(e)}_x$,  for the law of the process $\xi$ starting from $x\in \mathbb{R}$ and when $x=0$, we use the notation $\mathbb{P}^{(e)}$ for $\mathbb{P}^{(e)}_0$.  
 	
 		Further, under  condition \eqref{eq_finitemean}, the process $\left(Z_t e^{-\xi_t},  t\geq 0\right)$ is a quenched martingale implying that  for any $t\geq 0$ and $z\geq 0$,
 	\begin{equation}\label{martingquenched}
 		\mathbb{E}_z[Z_t \ \vert \ S]=ze^{\xi_t}, \ \qquad \mathbb{P}_z \ \textrm{-a.s},
 	\end{equation}
 	see Bansaye et al. \cite{bansaye2021extinction}.   In other words, the process $\xi$
 	plays an analogous role as the random walk associated to the logarithm  of  the mean of the offsprings  in the discrete time framework
 	and  leads to the usual classification for the long-term behaviour of branching processes. More precisely, we say that the 
 	process $Z$ is subcritical,  critical or supercritical accordingly as  $\xi$ drifts to $-\infty$, oscillates or drifts to $+\infty$. %In this manuscript, we focus on the subcritical regime and more precisely, in the weakly subcritical regime. 
 	
 	In addition, under  condition \eqref{eq_finitemean}, there is  another quenched martingale associated to  $(Z_t e^{-\xi_t}, t\geq 0)$ which allow us to compute its Laplace transform, see for instance Proposition 2 in \cite{palau2018branching} or  Theorem 3.4 in \cite{he2018continuous}. In order to compute the Laplace transform of $Z_t e^{-\xi_t}$, we first introduce  the unique  positive solution $(v_t(s,\lambda, \xi), s\in [0,t])$ of the following backward differential equation 
 	\begin{equation}\label{eq_BDE}
 		\frac{\partial}{\partial s} v_t(s,\lambda, \xi) = e^{\xi_s} \psi_0(v_t(s, \lambda, \xi)e^{-\xi_s}), \qquad v_t(t,\lambda, \xi) = \lambda,
 	\end{equation}
 	where 
 	\begin{equation}\label{eq_phi0}
 		\psi_0(\lambda)= \psi(\lambda)- \lambda \psi'(0+)=\varrho^2 \lambda^2 + \int_{(0,\infty)} \big(e^{-\lambda x} - 1 + \lambda x\big) \mu(\mathrm{d} x) .
 	\end{equation}
 	Then the process $\left(\exp\{-v_t(s,\lambda, \xi) Z_s e^{-\xi_s}\},  0\le s\le t\right)$ is a quenched martingale implying that for any $\lambda\geq 0$ and $t\geq s\geq 0$, 
 	\begin{equation}\label{eq_Laplace}
 		\mathbb{E}_{(z,x)}\Big[\exp\{-\lambda Z_t e^{-\xi_t}\}\  \Big|\, S, \mathcal{F}^{(b)}_s\Big] = \exp\{-Z_se^{-\xi_s}v_t(s, \lambda, \xi)\}.
 	\end{equation}
We may think of $v_t(\cdot, \cdot, \xi)$ as an inhomogeneous cumulant semigroup determined by the time-dependent branching mechanism $(s,\theta)\mapsto e^{\xi_s} \psi_0(\theta e^{-\xi_s})$. The functional $v_t(\cdot, \cdot, \xi)$ is quite involved, except for a few cases (stable and Nevue cases), due to the stochasticity coming from the time-dependent branching mechanism which makes it  even not so easy to control. 
 % Moreover, we also consider the random semigroup $h_{s,t}(\lambda)=e^{-\xi_s}v_t(s,\lambda e^{\xi_t},\xi)$ which is well defined for all $\lambda\geq 0$ and $s\in [0,t]$ and satisfies
 	%\begin{equation}\label{eq_transLaplace}
 		%\mathbb{E}_z\left[ e^{-\lambda Z_t}\Big|\, S,  \mathcal{F}^{(b)}_s \right] = \exp\left\{- Z_sh_{s,t}(\lambda)\right\},
 	%\end{equation}
 %see Theorem 3.4 in \cite{he2018continuous}. According to Section 2 in \cite{he2018continuous}, the mapping $s\mapsto h_{s,t}(\lambda)$ is the  unique positive pathwise solution to the integral differential equation 
 	%\begin{equation}\label{eq_backward}
 		%h_{s,t}(\lambda)= e^{\xi_t -\xi_s}\lambda - \int_{s}^{t}e^{\xi_r-\xi_s}\psi_0\big(h_{r,t}(\lambda)\big)\mathrm{d} r, \quad \quad 0\leq s\leq t.
 	%\end{equation}

 	In the what follows, we assume that $\xi$ is not a compound Poisson process to avoid the possibility  that  the process visits the same maxima or minima at distinct times which can make our analysis more involved. Moreover, we also require the following exponential moment condition, 
 	\begin{equation}\label{eq_moments}\tag{\bf{H1}}
 		\textrm{there exists }\quad  \vartheta > 1\ \text{ such that } \  \int_{\{|x|>1\}}e^{\lambda x}\pi(\mathrm{d} x)<\infty, \quad  \textrm{for all} \quad \lambda \in [0, \vartheta],
 	\end{equation}
 which is equivalent to the existence of  the Laplace transform on $[0, \vartheta]$, i.e. $\mathbb{E}^{(e)}[e^{\lambda \xi_1}]$ is well defined for $\lambda\in [0, \vartheta]$ (see for instance Lemma 26.4 in Sato \cite{ken1999levy}). The latter implies that we can  introduce  the Laplace exponent of $\xi$  as follows 
 \[
 \Phi_\xi(\lambda):=\log \mathbb{E}^{(e)}[e^{\lambda \xi_1}], \qquad \textrm{ for }\quad \lambda\in [0, \vartheta]. 
 \]
 Again from Lemma 26.4  in \cite{ken1999levy}, we also have $\Phi_\xi(\lambda)\in C^\infty$ and $\Phi_\xi^{\prime\prime}(\lambda)>0$, for $\lambda\in (0, \vartheta)$.
 	
 	Another object which will be relevant in our analysis  is the so-called exponential martingale associated to the L\'evy process $\xi$, i.e.
 	\[
 	M^{(\lambda)}_t=\exp\Big\{\lambda\xi_{t}-t\Phi_\xi(\lambda)\Big\}, \qquad t\ge 0,
 	\]
 	which is well-defined for $\lambda\in [0,\vartheta]$ under assumption \eqref{eq_moments}. It is well-known that $(M^{(\lambda)}_t, t\ge 0)$ is  a $(\mathcal{F}^{(e)}_t)_{t\ge 0}$-martingale and that it induces a change of measure which is known as the Esscher transform, that is to say
 	\begin{equation}\label{eq_medida_Essher}
 		\mathbb{P}^{(e,\lambda)}(\Lambda):= \mathbb{E}^{(e)}\Big[M_t^{(\lambda)} \mathbf{1}_{\Lambda}\Big], \qquad \textrm{for}\quad \Lambda\in \mathcal{F}^{(e)}_t.
 	\end{equation}
	Let us introduce the  dual process $\widehat{\xi}=-\xi$  which is also a L\'evy process satisfying that for any fixed time $t>0$, the processes 
 	\begin{equation}\label{eq_lemmaduality}
 		(\xi_{(t-s)^-}-\xi_{t}, 0\le s\le t)\qquad  \textrm{and}\qquad (\widehat{\xi}_s, 0\le s\le t),
 	\end{equation}
 	have the same law, with the convention that $\xi_{0^-}=\xi_0$ (see for instance Lemma 3.4 in Kyprianou \cite{kyprianou2014fluctuations}). For every $x\in \mathbb{R}$,  let $\widehat{\mathbb{P}}_x^{(e)}$ be the law of $x+\xi$ under $\widehat{\mathbb{P}}^{(e)}$, that is the law of $\widehat{\xi}$ under $\mathbb{P}_{-x}^{(e)}$.  We also  introduce the running infimum  and supremum  of $\xi$, by
 	\begin{equation*}
 		\underline{\xi}_t = \inf_{0\leq s\leq t} \xi_s \qquad \textrm{ and } \qquad \overline{\xi}_t = \sup_{0 \leq s \leq t} \xi_s, \qquad \textrm{for} \qquad t \geq 0.
 	\end{equation*}
Similarly to the critical case, which was studied by  Bansaye et al. \cite{bansaye2021extinction}, the asymptotic analysis  of the weakly subcritical regime requires the notion of the renewal functions $U^{(\lambda)}$ and $\widehat{U}^{(\lambda)}$ under $\mathbb{P}^{(e,\lambda)}$, which are associated to the supremum and infimum of $\xi$, respectively. See Section \ref{preLevy} for a proper definition (or the references therein).
 	
 	 For our purposes, we also require the notion of conditioned  L\'evy processes  and continuous state branching processes in a conditioned L\'evy environment.  	 Let us define the probability $\mathbb{Q}_{x}$ associated  to the L\'evy process $\xi$ started at $x>0$ and killed  at time $\zeta$ when it 
 	first enters $(-\infty, 0)$, that is to say
 	$$ \mathbb{Q}_{x}\big[f(\xi_t)\mathbf{1}_{\{\zeta>t\}}\Big ]:= \mathbb{E}^{(e)}_{x}\Big[f(\xi_t)\mathbf{1}_{\left\{\underline{\xi}_t> 0\right\}}\Big], $$
 	where $f:\mathbb{R}_+\to \mathbb{R}$ is a measurable function.  
 	
 	According to Chaumont and Doney \cite[Lemma 1]{chaumont2005levy}, under the assumption that $\xi$ does not drift towards $-\infty$,  we have that the renewal function $\widehat{U}:=\widehat{U}^{(0)}$ is invariant for the killed process. In other words,   for all $x> 0$ and $t\ge 0$,
 	\begin{equation}
 		\label{eq_fctharm}
 		\mathbb{Q}_x\left[\widehat{U}(\xi_t)\mathbf{1}_{\{\zeta>t\}}\right]=\mathbb{E}^{(e)}_x\left[\widehat{U}(\xi_t)\mathbf{1}_{\left\{\underline{\xi}_t> 0\right\}}\right]=\widehat{U}(x).
 	\end{equation}
 	Hence, from the  Markov property, we deduce that  $(\widehat{U}(\xi_t)\mathbf{1}_{\{\underline{\xi}_t>0\}}, t\geq 0)$ is a martingale with respect to $(\mathcal{F}_t^{(e)})_{t\geq 0}$. We may now use this martingale to define a change of measure corresponding to the law of $\xi$ \textit{conditioned to stay positive} as a Doob-$h$ transform. %Before doing so,  let us recall that   $\xi$ is adapted to the filtration $(\mathcal{F}^{(e)}_t)_{t\ge 0}$.
 	Under the assumption that $\xi$  does not drift towards $-\infty$, the law of the process $\xi$ conditioned to stay positive is defined as follows, for $\Lambda \in \mathcal{F}^{(e)}_t$ and $x>0$, 
 	\begin{equation} \label{defPuparrowx}
 		\mathbb{P}^{(e),\uparrow}_{x} (\Lambda)
 		:=\frac{1}{\widehat{U}(x)}\mathbb{E}^{(e)}_{x}\left[\widehat{U}(\xi_t) \mathbf{1}_{\left\{\underline{\xi}_t> 0\right\}} \mathbf{1}_{\Lambda}\right].
 	\end{equation}
 	On the other hand, by duality,  under the assumption that $\xi$  does not drift towards $\infty$, the law of the process $\xi$ \textit{conditioned to stay negative} is defined for $x<0$, as follows
 	\begin{equation} \label{defdownarrowx}
 		\mathbb{P}^{(e),\downarrow}_{x} (\Lambda)
 		:=\frac{1}{U(-x)}\mathbb{E}^{(e)}_{x}\left[U(-\xi_t) \mathbf{1}_{\left\{\overline{\xi}_t<0\right\}} \mathbf{1}_{\Lambda}\right].
 	\end{equation}
L\'evy processes conditioned to stay positive (and negative) are well studied objects. For a complete overview of this theory the reader is referred to \cite{bertoin1996levy, chaumont1996conditionings, chaumont2005levy} and references therein.

 	Similarly to the definition of L\'evy processes conditioned to stay positive (and negative) given above,  %and following a similar strategy as in the discrete framework in Afanasyev et al. \cite{afanasyev2005criticality}, 
	we may  introduce a continuous state branching processes in a L\'evy environment conditioned to stay positive as a Doob-$h$ transform. The aforementioned process was first investigated by Bansaye et al.  \cite{bansaye2021extinction} with the aim to study the survival event in a critical L\'evy environment. In other words, they proved the following result.
 	\begin{lemma}[Bansaye et. al. \cite{bansaye2021extinction}]\label{teo_bansayemtg}
 		Let us assume that $z,x >0$. Under the law $\mathbb{P}_{(z,x)}$, the process $(\widehat{U}(\xi_t)\mathbf{1}_{\{\underline{\xi}_t> 0\}}, t\geq 0)$ is a martingale with respect to $(\mathcal{F}_t)_{t\geq 0}$. Moreover the following Doob-$h$ transform holds, for $\Lambda \in \mathcal{F}_t$, 
 		\begin{equation*}
 			\mathbb{P}^{\uparrow}_{(z,x)}(\Lambda):=\frac{1}{\widehat{U}(x)}\mathbb{E}_{(z,x)}\big[\widehat{U}(\xi_t)\mathbf{1}_{\{\underline{\xi}_t> 0\}}\mathbf{1}_\Lambda\big],
 		\end{equation*}
 		defines   a continuous state branching process   in a L\'evy environment $\xi$ conditioned to stay positive.
 	\end{lemma}
 	Furthermore,  appealing to  duality and Lemma \ref{teo_bansayemtg}, we may deduce that, under $\mathbb{P}_{(z,x)}$ with $z>0$ and $x<0$,  the process $(U(-\xi_t)\mathbf{1}_{\{\overline{\xi}_t< 0\}}, t\geq 0)$ is a martingale with respect to $(\mathcal{F}_t)_{t\geq 0}$. Hence,  the law of \textit{ continuous state branching processes in a L\'evy environment $\xi$ conditioned to stay negative} is defined as follows: for $z>0$, $x<0$ and $\Lambda \in \mathcal{F}_t$, 
 	%\begin{equation*}
 	%\p^{\downarrow}_{(z,x)}(\Lambda):=\frac{1}{U(-x)}\widehat{\e}_{(z,-x)}[U(\xi_t)\mathbf{1}_{\{\underline{\xi}_t\geq 0\}}\mathbf{1}_\Lambda].
 	
 	\begin{equation}\label{eq_CDBPnegativo}
 		\mathbb{P}^{\downarrow}_{(z,x)}(\Lambda): =\frac{1}{U(-x)}\mathbb{E}_{(z,x)}\big[U(-\xi_t)\mathbf{1}_{\{\overline{\xi}_t< 0\}}\mathbf{1}_\Lambda\big].
 	\end{equation}
 	%\frac{1}{U(-x)}\e_{(z,x)}[U(-\xi_t)\mathbf{1}_{\{\underline{-\xi}_t\geq 0\}}\mathbf{1}_\Lambda]
 	%This change of measure is the well-known Doob $h$-transform from the theory of Markov processes.  
 	%Intuitively speaking, $\mathbb{P}^{\uparrow}_{(z,x)}$ and $\mathbb{P}^{\downarrow}_{(z,x)}$ correspond to the law of $(Z,\xi)$ conditioning the random environment $\xi$ not to enter $(-\infty,0)$ and $(0,\infty)$, respectively.

 Recall that we are interested in the probability of survival under the weakly subcritical regime,   that is  \eqref{eq_moments} is satisfied and  the Laplace exponent of $\xi$ is such that  
 \[
 \Phi_{\xi}'(0)<0< \Phi_{\xi}'(1) \textrm{ and  there exists  } \gamma \in (0,1) \textrm{ which solves } \Phi_{\xi}'(\gamma)=0. 
 \]
In other words, the L\'evy process $\xi$  drifts to $-\infty$ a.s., under $\mathbb{P}^{(e)}$, and to $+\infty$ a.s., under $\mathbb{P}^{(e,1)}$. In the remainder of this manuscript, we will always assume that  the process $Z$ is in the weakly subcritical regime.

Our first main result requires  that  the branching mechanism $\psi$ is regularly varying at $0$, that is there exist $\beta\in (0,1]$
\begin{equation}\label{betacond}\tag{\bf{H2}}
\psi_0(\lambda)=\lambda^{1+\beta}\ell(\lambda),
\end{equation}
where $\ell$ is a slowly varying function at $0$. See Bingham et al. \cite{bingham1989regular} for a proper definition.

For simplicity on exposition, we introduce the function $\kappa^{(\lambda)}(0,\theta)$ as follows
\[
\int_0^\infty e^{-\theta y} U^{(\gamma)}(y)  \mathrm{d} y=\frac{1}{\theta\kappa^{(\gamma)}(0,\theta)}, \qquad \theta> 0.
\]
\begin{theorem}\label{prop_weakly_pareja}
	Let $x,z>0$. Assume that $Z$ is weakly subcritical and that condition \eqref{betacond} holds, hence the random variable $\mathcal{U}_t:=Z_te^{-\xi_t}$  converges in distribution to some random variable $Q$ with values in $[0,\infty)$ as $t\to \infty$, under $\mathbb{P}_{(z,x)}\big(\cdot \ | \ \underline{\xi}_t>0 \big)$. Moreover, 
	\begin{equation}\label{eq_constantb1}
\mathfrak{b}(z,x):= \lim\limits_{t\to \infty}\mathbb{P}_{(z,x)}\Big(Z_t>0 \ \big|\big. \ \underline{\xi}_t>0\Big)>0,
	\end{equation}
	where
	\begin{equation*}
		\mathfrak{b}(z,x) =  1-  \lim\limits_{\lambda \to \infty} \lim\limits_{s\to \infty}  \int_{0}^\infty \int_{0}^{1} \int_{0}^{\infty} w^{\textbf{u}} \mathbb{P}^{(\gamma),\uparrow}_{(z,x)}\big(\mathcal{U}_s \in \mathrm{d} \textbf{u}\big) \mathbb{P}^{(e,\gamma),\downarrow}_{-y}\Big(\widehat{W}_s(\lambda) \in \mathrm{d} w\Big)\mu_\gamma(\mathrm{d}  y),	\end{equation*}
	with
	\begin{equation}\label{eq:mugamma}
	\widehat{W}_{s}(\lambda):= \exp\left\{-v_s(0,\lambda,\widehat{\xi})\right\}\quad\textrm{and}\quad \mu_\gamma(\mathrm{d} y) := \gamma\kappa^{(\gamma)}(0,\gamma)e^{-\gamma y} U^{(\gamma)}(y) \mathbf{1}_{\{y>0\}} \mathrm{d} y.
	\end{equation}
\end{theorem}
It is important to note that in general, it seems difficult to compute explicitly the constant $\mathfrak{b}(z,x)$ except for the stable case. In the stable case, we observe that the constant $\mathfrak{b}(z,x)$ is given  in terms of two independent exponential functionals of conditioned L\'evy processes. Denote by $\texttt{I}_{s,t}(\beta \xi)$  the exponential functional of  the L\'evy process $\beta \xi$, i.e., 
	\begin{equation}\label{eq_expfuncLevy}
		\texttt{I}_{s,t}(\beta \xi):=\int_{s}^{t}e^{-\beta \xi_u} \mathrm{d} u, \quad \quad  0\leq s\leq t. 
	\end{equation}
Hence, when $\psi_0(\lambda)=C\lambda^{1+\beta}$ with $C>0$ and $\beta\in(0,1)$, we have 
\begin{equation*}
		\mathfrak{b} (z,x) = \gamma\kappa^{(\gamma)}(0,\gamma) \int_{0}^{\infty} e^{-\gamma y} U^{(\gamma)} (y) G_{z,x}(y){\rm d} y,
\end{equation*}
where 
\begin{equation}\label{functG}
\begin{split}
G_{z,x}(y)&:=   \int_{0}^{\infty} \int_{0}^{\infty} \left(1-e^{-ze^{-x}(\beta C w +\beta C u)^{-1/\beta}}\right)\\
&\hspace{4cm}\mathbb{P}^{(\gamma),\uparrow}_{(z,x)}\big(\texttt{I}_{0,\infty}(\beta \xi) \in \mathrm{d} w\big) \mathbb{P}^{(e,\gamma),\downarrow}_{-y}\Big(\texttt{I}_{0,\infty}(\beta \widehat{\xi}) \in \mathrm{d} u\Big).	
% G_{z,x}(y):= \mathbb{Q}_{(x,y)}^{(e, \gamma)}\left[1-\exp\left\{-ze^{-x}\left(\beta C 	\texttt{I}_{0,\infty}(\beta \xi) + \beta C 	\texttt{I}_{0,\infty}(\beta \widehat{\xi})\right)^{-1/\beta}\right\}\right]  .
 \end{split}
 \end{equation}	
We refer to subsection \ref{sec:stable} for further details about the computation of this constant.

Under the assumption that  $Z$ is  weakly subcritical, the running infimum of the auxiliary process $\xi$ satisfies the following asymptotic behaviour:  for $x>0$, 
\begin{equation}\label{eq_weakly_cota_hirano}
	\mathbb{P}^{(e)}_x\left(\underline{\xi}_t> 0\right) \sim  \frac{A_\gamma}{\gamma\kappa^{(\gamma)}(0,\gamma)} e^{\gamma x} \widehat{U}^{(\gamma)}(x) t^{-3/2}e^{\Phi_{\xi}(\gamma)t}, \quad \text{as} \qquad t \to \infty,
\end{equation}
where 
\begin{equation}\label{eq_weakly_constante}
	A_\gamma:= \frac{1}{\sqrt{2 \pi \Phi_{\xi}''(\gamma)}} \exp\left\{\int_{0}^{\infty}(e^{-t}-1)t^{-1}e^{-t\Phi_{\xi}(\gamma)}\mathbb{P}^{(e)}(\xi_t = 0) \mathrm{d} t\right\},
\end{equation}
see for instance  Lemma A in \cite{hirano2001levy} (see also Proposition 4.1 in \cite{li2018asymptotic}). Such asymptotic turns out to be the leading term in the asymptotic behaviour of the probability of survival as it is stated below.

\begin{theorem}[Weakly subcritical regime]\label{teo_debil}
	Let $z>0$. Assume that $Z$ is weakly subcritical and that the slowly varying function  in   \eqref{betacond}  satisfies that there exists a constant $C>0$, such that $\ell(\lambda)>C$.  
	Then  there exists $0<\mathfrak{B}(z)<\infty$ such that
	\[
	\begin{split}
		\lim\limits_{t\to \infty}t^{-3/2}e^{-\Phi_{\xi}(\gamma)t} \mathbb{P}_{z}(Z_t>0) = \mathfrak{B}(z),
	\end{split}
	\]
	with
	$$ \mathfrak{B}(z):=\frac{A_\gamma }{ \gamma\kappa^{(\gamma)}(0,\gamma)}\lim_{x\to \infty} \mathfrak{b}(z,x)e^{\gamma x}\widehat{U}^{(\gamma)}(x), 
	$$
	where $\mathfrak{b}(z,x)$ and $A_{\gamma}$ are the constants defined in \eqref{eq_constantb1} and \eqref{eq_weakly_constante}, respectively.
\end{theorem} 	

It is important to note that in the stable case,  the constant $\mathfrak{B}(z)$  coincides with the constant that appears  in Theorem 5.1 in Li and Xu \cite{li2018asymptotic},  that is
\begin{equation*}
	\begin{split}
		\mathfrak{B}(z)=    A_\gamma \lim\limits_{x\to \infty} e^{\gamma x} \widehat{U}^{(\gamma)}(x)\int_{0}^{\infty} e^{-\gamma y} U^{(\gamma)}(y)G_{z,x}(y){\rm d} y,
	\end{split}
\end{equation*}
where $G_{z,x}$ is defined in \eqref{functG}.

{\bf Some comments about our results:}
We first remark that our  assumption \eqref{betacond} clearly implies 		
		\begin{equation}\label{eq_xlog2x}
			\int^\infty x \log^2 x \mu(\mathrm{d} x) <\infty.
		\end{equation}
		The latter condition was used before in Proposition 3.4 in \cite{bansaye2021extinction} to control the effect of a favourable environment on the event of survival. Unlike the critical case, in the weakly subcritical regime the slightly stronger condition \eqref{betacond} is required to guarantee the convergence in Theorem \ref{prop_weakly_pareja}, which allows us to have a good control of the event of survival given  favourable environments. A crucial ingredient in Theorem \ref{prop_weakly_pareja} is an extension of a sort of functional limit theorem for  conditioned  L\'evy and CBLE processes (see Proposition \ref{prop_weakly_primera} below). More precisely, we would require the asymptotically independence of the processes $((Z_u, \xi_u), 0\leq u \leq r \ | \  \underline{\xi}_t>0)$ and $(\xi_{(t-u)^{-}}, 0 \leq u \leq \delta t  \ | \  \underline{\xi}_t>0)$ as $t $ goes to $\infty$, for every $r, t\geq 0$ and $\delta\in (0,1)$. We claim that this result must be true in full generality (in particular Theorem \ref{prop_weakly_pareja} under \eqref{eq_xlog2x}) since it holds for  random walks (see Theorem 2.7 in \cite{afanasyev2012}) but it seems  not so easy to deduce. Meanwhile in the discrete setting the result follows directly from duality, in the L\'evy case the convergence will depend on a much deeper analysis on the asymptotic behaviour for bridges of L\'evy processes and their conditioned version.   It seems that  a better understanding of  conditioned L\'evy bridges is required.

	On the other hand, the condition that the slowly varying function $\ell$ is bounded from below is required to control the absorption event under unfavourable environments (see Lemma \ref{lem_cero_weakly}) and to guarantee a.s. absorption.  Indeed,  under Grey's condition 
	\begin{equation}\label{GreysCond}
		\int^{\infty}\frac{1}{\psi_0(\lambda)} \mathrm{d} \lambda < \infty,
	\end{equation} 
	 and equation \eqref{eq_Laplace}, we deduce  that for $z,x>0$
\begin{equation}\label{cotasup} 
\mathbb{P}_{(z,x)}\Big(Z_t>0,\  \underline{\xi}_t\le -y\Big)=\mathbb{E}^{(e)}\left[\left(1-e^{-z v_t(0,\infty, \xi)}\right)\mathbf{1}_{\{\underline{\xi}_t\le -y-x\}}\right],  \quad \textrm{ for } \quad y\ge 0, 
\end{equation}
where $v_t(0,\infty, \xi)$ is $\mathbb{P}^{(e)}$-a.s. finite for all $t\ge 0$, (see Theorem 4.1 and Corollary 4.4 in \cite{he2018continuous}) but perhaps equals 0. We note that \eqref{eq_xlog2x} (and implicitly \eqref{betacond})  guarantees that $v_t(0,\infty, \xi)>0$,  $\mathbb{P}^{(e)}$-a.s. for all $t>0$ (see for instance Proposition 3 in \cite{palau2018branching}). Since the functional $v_t(0,\infty, \xi)$ depends strongly on the environment, it seems difficult to estimate the right-hand side of \eqref{cotasup}. Actually, it seems not so easy to  obtain a sharp control of \eqref{cotasup}. Condition \eqref{betacond}  implies that Grey's condition 
 is fulfilled  and the assumption  that $\ell$ is bounded from below allow us  to upper bound \eqref{cotasup} in terms of the exponential functional of $\xi$.

Finally, we point out that in the discrete setting such  probability can be estimated directly in terms of the infimum of the environment since the event of survival is equal to the event that the current population is bigger or equal to one, something that cannot be performed in our setting.

The remainder of this paper is devoted to the proof of the main results.

\section{Proofs}\label{sec_absweakly}
This section is devoted to the proofs of our main results and the computation of the constant  $\mathfrak{b} (z,x) $ in the stable case. We start with some preliminaries on L\'evy processes.
 	\subsection{Preliminaries on L\'evy processes}\label{preLevy}

	Recall that  $\mathbb{P}^{(e)}_x$  denotes the law of the L\'evy process $\xi$ 
 	starting from $x\in \mathbb{R}$ and when $x=0$, we use the notation $\mathbb{P}^{(e)}$ for $\mathbb{P}^{(e)}_0$. We also recall that  $\widehat{\xi}=-\xi$ denotes the dual process and denote by  $\widehat{\mathbb{P}}_x^{(e)}$ for its law starting at  $x\in \mathbb{R}$.

 	In what follows, we require the notion of the 
 	reflected processes $\xi-\underline{\xi}$ and  $\overline{\xi}-\xi$  which are Markov processes with respect to the  filtration $(\mathcal{F}^{(e)}_t)_{t\geq 0}$ and whose semigroups
 	satisfy the Feller property (see for instance Proposition VI.1 in the monograph of Bertoin \cite{bertoin1996levy}).  We 
 	denote by $L=(L_t, t \geq 0 )$ and $\widehat{L}=(\widehat{L}_t, t \geq 0 )$  for  the local times of $\overline{\xi}-\xi$ and $\xi-\underline{\xi}$ at $0$, respectively,  in the sense of Chapter IV in \cite{bertoin1996levy}. If $0$ is regular for $(-\infty,0)$ or regular downwards, i.e.
 	\[
 	\mathbb{P}^{(e)}(\tau^-_0=0)=1,
 	\] 
 	where $\tau^{-}_0=\inf\{s> 0:  \xi_s\le 0\}$, then $0$ is regular for the reflected process $\xi-\underline{\xi}$ and then, up to a multiplicative constant, $\widehat{L}$ is the unique additive functional of the reflected process whose set of increasing points is $\{t:\xi_t=\underline{\xi}_t\}$. If $0$ is not regular downwards then the set $\{t:\xi_t=\underline{\xi}_t\}$ is discrete and we define the local time $\widehat{L}$ as the counting process of this set. The same properties holds for $L$ by duality.
 	
 	Let us denote by $L^{-1}$ and $\widehat{L}^{-1}$ the right continuous inverse of  $L$ and $\widehat{L}$, respectively.  The range of the inverse local times  $L^{-1}$ and $\widehat{L}^{-1}$, correspond to the sets of real times at which new maxima and new minima occur, respectively. Next, we introduce the so called increasing ladder height process by
 	\begin{equation}\label{defwidehatH}
 		H_t=\overline{\xi}_{L_t^{-1}}, \qquad t\ge 0.
 	\end{equation}
 	The pair $(L^{-1}, H)$ is a bivariate subordinator, as is the case of  the pair $(\widehat{L}^{-1}, \widehat{H})$ with
 	\[
 	\widehat{H}_t=-\underline{\xi}_{\widehat{L}_t^{-1}}, \qquad t\ge 0.
 	\]
 	The range of the process $H$ (resp. $\widehat{H}$) corresponds to the set of new maxima (resp. new minima). Both pairs are known as descending and ascending ladder processes, respectively.   	
 	
 	We also recall that   $U^{(\lambda)}$ and $\widehat{U}^{(\lambda)}$ denote the renewal functions under $\mathbb{P}^{(e,\lambda)}$. Such functions are defined as follows:   for all $x>0$, 
 	\begin{equation}\label{eq_Utheta}
 		U^{(\lambda)}(x) := \mathbb{E}^{(e,\lambda)}\left[\int_{[0,\infty)} \mathbf{1}_{\left\{\overline{\xi}_t\leq x\right\}} \mathrm{d} L_t\right] \quad 
 		\textrm{and}\quad
 		\widehat{U}^{(\lambda)}(x) := \mathbb{E}^{(e,\lambda)}\left[\int_{[0,\infty)} \mathbf{1}_{\left\{\underline{\xi}_t\geq -x\right\}} \mathrm{d} \widehat{L}_t\right].
 	\end{equation}
 	The renewal functions $U^{(\lambda)}$ and $\widehat{U}^{(\lambda)}$ are  finite, subadditive, continuous and increasing. Moreover, they are identically 0 on $(-\infty, 0]$,  strictly positive on $(0,\infty)$  and satisfy 
 	\begin{equation}
 		\label{grandO}
 		U^{(\lambda)}(x)\leq C_1 x \qquad\textrm{and}\qquad \widehat{U}^{(\lambda)}(x)\leq C_2 x \quad \text{  for any } \quad x\geq 0,
 	\end{equation}
 	where $C_1, C_2$ are finite constants  (see for instance  Lemma 6.4 and Section 8.2 in the monograph of Doney 
 	\cite{doney2007fluctuation}). Moreover $U^{(\lambda)}(0)=0$ if $0$ is regular upwards  and $U^{(\lambda)}(0)=1$ otherwise, similalry $\widehat{U}^{(\lambda)}(0)=0$ if $0$ is regular upwards  and $\widehat{U}^{(\lambda)}(0)=1$ otherwise.
 
 	Furthermore, it is important to note that by a simple change of variables, we can rewrite the renewal functions $U^{(\lambda)}$ and $\widehat{U}^{(\lambda)}$ in terms of the ascending and descending ladder height processes. Indeed,  the measures induced by $U^{(\lambda)}$ and  $\widehat{U}^{(\lambda)}$ can be rewritten as follows,
 	\begin{equation*}
 	U^{(\lambda)}( x)=\mathbb{E}^{(e,\lambda)}\left[\int_{0}^\infty\mathbf{1}_{\{{H}_t \le x\}}\mathrm{d} t\right]\qquad \textrm{and}\qquad \widehat{U}^{(\lambda)}(x)=\mathbb{E}^{(e,\lambda)}\left[\int_{0}^\infty\mathbf{1}_{\{\widehat{H}_t \le x\}}\mathrm{d} t\right].
	 \end{equation*}
 	Roughly speaking, the renewal function $U^{(\lambda)}(x)$ (resp.  $\widehat{U}^{(\lambda)}(x)$)  ``measures'' the amount of time that the ascending (resp. descending) ladder height process spends on the interval $[0,x]$ and in particular induces a measure on $[0,\infty)$ which is known as the renewal measure.  The latter implies
 \begin{equation}\label{bivLap}
 	\int_{[0,\infty)} e^{-\theta x} U^{(\lambda)}( x)\mathrm{d}x = \frac{1}{\theta\kappa^{(\lambda)}(0,\theta)},  \qquad  \theta>0,
\end{equation}
where $\kappa^{(\lambda)}(\cdot,\cdot)$ is the bivariate Laplace exponent of the ascending ladder process $(L^{-1}, H)$, under $\mathbb{P}^{(e,\lambda)}$ (see for instance \cite{bertoin1996levy, doney2007fluctuation, {kyprianou2014fluctuations}}).

\subsection{Proof of Theorem \ref{prop_weakly_pareja}}
Our arguments follows a similar strategy as  in Afanasyev et al.  \cite{afanasyev2012} where the discrete setting is considered. Although the matter of considering continuous time leads to significant changes such as that $0$ might be polar.  Our first proposition is the continuous analogue of Proposition 2.5 in  \cite{afanasyev2012} and  in some sense it is  a generalisation of Theorem 2 part (a) in Hirano \cite{hirano2001levy} (see also Proposition 4.2 in \cite{li2018asymptotic}). In particular, the result tell us that, for every $r, t\geq 0$ and $s\leq t$, the conditional processes $((Z_u, \xi_u), 0\leq u \leq r \ | \  \underline{\xi}_t>0)$ and $(\xi_{(t-u)^{-}}, 0 \leq u \leq s  \ | \  \underline{\xi}_t>0)$ are asymptotically independent as $t\to \infty$.  %\nat{This result is a key point to find the Laplace transform of the random variable $ Z_{s}e^{-\xi_{s}}$ under the measure $\mathbb{P}_{(z,x)}\big(\cdot \ | \ \underline{\xi}_t>0 \big)$. }
	
Before we state our first result in this subsection, we recall that   $\mathbb{D}([0,t])$ denotes the space of c\`adl\`ag real-valued functions on $[0,t]$.

\begin{proposition}\label{prop_weakly_primera} 
	Let  $f$ and $g$ be  continuous and bounded functionals on $\mathbb{D}([0, t])$.  We also  set \ $\mathcal{U}_r := g((Z_u, \xi_u), 0\leq u \leq  r)$, and for $s\le t$
	\[
	\widehat{W}_s := f(-\xi_u, 0\leq u \leq  s), \qquad \textrm{and}\qquad\widetilde{W}_{t-s,t} := f(\xi_{(t-u)^-}, 0\leq u \leq  s).
	\]
	Then for any bounded continuous function $\varphi:\mathbb{R}^3\to \mathbb{R}$, we have
	\[
	\begin{split}
		\lim_{t\to\infty}&\frac{\mathbb{E}^{(\gamma)}_{(z,x)}\Big[\varphi(\mathcal{U}_{r},\widetilde{W}_{t-s, t},\xi_t)e^{-\gamma \xi_t}\mathbf{1}_{\{\underline{\xi}_{t} > 0\}}\Big]}{\mathbb{E}^{(e,\gamma)}_x\Big[e^{-\gamma \xi_t}\mathbf{1}_{\{\underline{\xi}_{t} > 0\}}\Big]}\\
		&\hspace{4cm}  =\iiint \varphi(u,v,y) \mathbb{P}^{(\gamma),\uparrow}_{(z,x)}\big(\mathcal{U}_r \in \mathrm{d} u\big)\mathbb{P}^{(e,\gamma),\downarrow}_{-y}\Big(\widehat{W}_s \in \mathrm{d} v\Big)\mu_\gamma(\mathrm{d} y),
	\end{split}
	\]
	with $$\mu_\gamma(\mathrm{d} y) := \gamma\kappa^{(\gamma)}(0,\gamma)e^{-\gamma y} U^{(\gamma)}(y) \mathbf{1}_{\{y>0\}} \mathrm{d} y.$$

\end{proposition}

\begin{proof} By a monotone class argument, it is enough to show the result for   continuous bounded functions of the form $\varphi(u,v,y)=\varphi_1(u)\varphi_2(v)\varphi_3(y)$, where  $\varphi_i: \mathbb{R} \to \mathbb{R}$ are  bounded and continuous functions, for $i=1,2,3$. That is, we will show that for $z,x>0$,
	\[
	\lim\limits_{t \to \infty}	\frac{\mathbb{E}^{(\gamma)}_{(z,x)}\Big[\varphi_1(\mathcal{U}_r)\varphi_2(\widetilde{W}_{t-s,t}) \varphi_3(\xi_t)e^{-\gamma \xi_t}\mathbf{1}_{\{\underline{\xi}_t > 0\}}\Big]}{\mathbb{E}^{(e,\gamma)}_x\Big[e^{-\gamma \xi_t} \mathbf{1}_{\{\underline{\xi}_t > 0\}}\Big]} =  \mathbb{E}_{(z,x)}^{(\gamma), \uparrow}[\varphi_1(\mathcal{U}_r)] \mathbb{E}_{\mu_\gamma}^{(e,\gamma), \downarrow}\Big[\varphi_2(\widehat{W}_s)\varphi_3(\xi_0)\Big],
	\]
	where 	\begin{equation}\label{eq_mugamma}
		\mathbb{E}_{\mu_\gamma}^{(e,\gamma), \downarrow}\Big[\varphi_2(\widehat{W}_s)\varphi_3(\xi_0)\Big] = \int_{(0, \infty)} \mathbb{E}_{-y}^{(e,\gamma), \downarrow}\Big[\varphi_2(\widehat{W}_s)\varphi_3(\xi_0)\Big]  \mu_\gamma (\mathrm{d}y).
	\end{equation}
For simplicity on exposition, we assume  $0\leq \varphi_i \leq 1$, for $i=1, 2, 3$.  We first observe from the Markov property that for $t\geq r+s$, we have
	\begin{equation}\label{eq_markov_prop1}
		\mathbb{E}^{(\gamma)}_{(z,x)}\Big[\varphi_1(\mathcal{U}_r)\varphi_2(\widetilde{W}_{t-s,t}) \varphi_3(\xi_t)e^{-\gamma \xi_t}\mathbf{1}_{\{\underline{\xi}_t > 0\}}\Big]= \mathbb{E}_{(z,x)}^{(\gamma)}\left[\varphi_1(\mathcal{U}_r) \Phi_{t-r}(\xi_r)\mathbf{1}_{\{\underline{\xi}_{r} > 0\}}\right], 
	\end{equation}
	where 
	\begin{equation}\label{eq_phiu}
		\Phi_{u}(y):= \mathbb{E}_{y}^{(e,\gamma)} \left[\varphi_2(\widetilde{W}_{u-s,u})\varphi_3(\xi_{u})e^{-\gamma \xi_{u}}\mathbf{1}_{\{\underline{\xi}_{u} > 0\}}\right], \qquad u \geq s,\,\, y>0.
	\end{equation}
	Using the last definition and once again the Markov property,  we deduce the following identity
	\begin{equation}\label{eq_expec_weakly}
		\Phi_{t-r}(y)= \mathbb{E}_{y}^{(e,\gamma)} \left[\Phi_{s}(\xi_{t-r-s})\mathbf{1}_{\{\underline{\xi}_{t-r-s} > 0\}} \right],\qquad y>0.
	\end{equation}
	On the other hand, by Lemma 1 in \cite{hirano2001levy}, we know that for $\delta >0$ and $t\geq v$, 
	\begin{equation*}
		\lim\limits_{t \to \infty} \frac{\mathbb{E}^{(e,\gamma)}_y\left[e^{-(\delta + \gamma)\xi_{t-v}}\mathbf{1}_{\{\underline{\xi}_{t-v} > 0\}}\right]}{\mathbb{E}^{(e,\gamma)}_x\left[e^{-\gamma \xi_t}\mathbf{1}_{\{\underline{\xi}_{t} > 0\}}\right]} = \frac{\widehat{U}^{(\gamma)}(y)}{\widehat{U}^{(\gamma)}(x) }\frac{\displaystyle\int_{0}^{\infty} e^{-( \delta + \gamma) z} U^{(\gamma)}(z) \mathrm{d} z}{\displaystyle\int_{0}^{\infty} e^{- \gamma z} U^{(\gamma)}(z) \mathrm{d}z}.
	\end{equation*}
	Then by the continuity Theorem for the Laplace transform and using identity \eqref{bivLap}, for $h$ bounded and  continuous $\mu_\gamma$-a.s., it follows 
	\begin{equation}\label{eq_lim_prop1}
		\lim\limits_{t \to \infty} \frac{\mathbb{E}^{(e, \gamma)}_y\left[h(\xi_{t-v}) e^{- \gamma \xi_{t-v}}\mathbf{1}_{\{\underline{\xi}_{t-v} > 0\}}\right]}{\mathbb{E}^{(e, \gamma)}_x\left[e^{-\gamma \xi_t}\mathbf{1}_{\{\underline{\xi}_{t} > 0\}}\right]} = \frac{\widehat{U}^{(\gamma)}(y)}{\widehat{U}^{(\gamma)}(x)} \int_{0}^{\infty} h(z)  \mu_\gamma (\mathrm{d} z).
	\end{equation}
	If $h$ is positive and continuous but not bounded, we can truncate the function $h$, i.e., fix $n\in \mathbb{N}$ and define $h_n(x):= h(x)\mathbf{1}_{\{h(x)\leq n \}}$. Then by \eqref{eq_lim_prop1}, we have
	%\begin{equation}\label{eq_weakly_trunc1}
	\[
	\begin{split}
		\liminf_{t\to \infty}\frac{\mathbb{E}^{(e,\gamma)}_y\left[h(\xi_{t-v}) e^{- \gamma \xi_{t-v}}\mathbf{1}_{\{\underline{\xi}_{t-v} > 0\}}\right]}{\mathbb{E}^{(e,\gamma)}_x\left[e^{-\gamma \xi_t}\mathbf{1}_{\{\underline{\xi}_{t} > 0\}}\right]} &\geq \liminf_{t\to\infty}\frac{\mathbb{E}^{(e,\gamma)}_y\left[h_n(\xi_{t-v}) e^{- \gamma \xi_{t-v}}\mathbf{1}_{\{\underline{\xi}_{t-v} > 0\}}\right]}{\mathbb{E}^{(e,\gamma)}_x\left[e^{-\gamma \xi_t}\mathbf{1}_{\{\underline{\xi}_{t} > 0\}}\right]} \\ &= \frac{\widehat{U}^{(\gamma)}(y)}{\widehat{U}^{(\gamma)}(x)} \int_{0}^{\infty} h_n(z)  \mu_\gamma (\mathrm{d} z).
	\end{split}
	\]
	%\end{equation}
	On the other hand, since $h_n(x) \to h(x)$ as $n \to \infty$, by Fatou's Lemma 
	%\begin{eqnarray}\label{eq_weakly_trunc2}
	\[
	\liminf_{n\to \infty} \int_{0}^{\infty} h_n(z)  \mu_\gamma (\mathrm{d} z) \geq \int_{0}^{\infty} h(z)  \mu_\gamma (\mathrm{d} z).
	\]
	%\end{eqnarray}
	Thus putting both pieces together, we get
	\begin{equation}\label{eq_liminf_pro1}
		\liminf_{t \to \infty} \frac{\mathbb{E}^{(e,\gamma)}_y\left[h(\xi_{t-v}) e^{- \gamma \xi_{t-v}}\mathbf{1}_{\{\underline{\xi}_{t-v} > 0\}}\right]}{\mathbb{E}^{(e,\gamma)}_x\left[e^{-\gamma \xi_t}\mathbf{1}_{\{\underline{\xi}_{t} > 0\}}\right]} \geq \frac{\widehat{U}^{(\gamma)}(y)}{\widehat{U}^{(\gamma)}(x)} \int_{0}^{\infty} h(z)  \mu_\gamma (\mathrm{d} z).
	\end{equation}
	We want to apply the previous inequality to the function $h(x)=\Phi_s(x)e^{\gamma x}$. To do so, we need to verify that $\Phi_{s}(\cdot)$ is a positive and $\mu_\gamma$-a.s.-continuous function. First, we observe that discontinuities of $\Phi_{s}(\cdot)$ correspond to discontinuities of the map 
	\[
	{\tt e}: y \mapsto \mathbb{P}^{(e,\gamma)}\big(\underline{\xi}_t > -y\big).
	\] Since ${\tt e}(\cdot)$ is bounded and monotone,  it has a countable number of discontinuities. Thus $\Phi_{s}(\cdot)$ is continuous almost everywhere with respect to the Lebesgue measure and therefore $\mu_\gamma$-a.s.
	
	Now, from \eqref{eq_expec_weakly} and \eqref{eq_liminf_pro1} with $v=r+s$ and  $h(x)=\Phi_s(x)e^{\gamma x}$, we have
	\begin{equation}\label{eq2_liminf_pro1}
		\begin{split}
			\liminf_{t \to \infty} \frac{\Phi_{t-r}(y)}{\mathbb{E}^{(e,\gamma)}_x\left[e^{-\gamma \xi_t}\mathbf{1}_{\{\underline{\xi}_{t} > 0\}}\right]}&= \liminf_{t \to \infty} \frac{\mathbb{E}^{(e,\gamma)}_y\left[\Phi_{s}(\xi_{t-v}) e^{\gamma \xi_{t-v}}e^{- \gamma \xi_{t-v}}\mathbf{1}_{\{\underline{\xi}_{t-v} > 0\}}\right]}{\mathbb{E}^{(e,\gamma)}_x\left[e^{-\gamma \xi_t}\mathbf{1}_{\{\underline{\xi}_{t} > 0\}}\right]} \\ 
			&= \liminf_{t \to \infty} \frac{\mathbb{E}^{(e,\gamma)}_y\left[h(\xi_{t-v}) e^{- \gamma \xi_{t-v}}\mathbf{1}_{\{\underline{\xi}_{t-v} > 0\}}\right]}{\mathbb{E}^{(e,\gamma)}_x\left[e^{-\gamma \xi_t}\mathbf{1}_{\{\underline{\xi}_{t} > 0\}}\right]} \\ &\geq   \frac{\widehat{U}^{(\gamma)}(y)}{\widehat{U}^{(\gamma)}(x)} \int_{0}^{\infty} \Phi_s(z)e^{\gamma z}  \mu_\gamma (\mathrm{d}z). 
		\end{split}
	\end{equation}
	In view of identity \eqref{eq_markov_prop1} and the above inequality, replacing  $y$ by $\xi_r$, we get  from Fatou's Lemma
	\begin{equation}\label{eq3_liminf_prop1}
		\begin{split}
			\liminf_{t \to \infty} \frac{\mathbb{E}^{(\gamma)}_{(z,x)}\left[\varphi_1(\mathcal{U}_r)\varphi_2(\widetilde{W}_{t-s,t})\varphi_3(\xi_t)e^{-\gamma \xi_t}\mathbf{1}_{\{\underline{\xi}_{t} > 0\}}\right] }{\mathbb{E}^{(e,\gamma)}_x\left[e^{-\gamma \xi_t} \mathbf{1}_{\{\underline{\xi}_{t} > 0\}}\right]}&=  \liminf_{t \to \infty} \frac{ \mathbb{E}_{(z,x)}^{(\gamma)}\left[\varphi_{1}(\mathcal{U}_r) \Phi_{t-r}(\xi_r)\mathbf{1}_{\{\underline{\xi}_{r} > 0\}}\right] }{\mathbb{E}^{(e,\gamma)}_x\left[e^{-\gamma \xi_t}\mathbf{1}_{\{\underline{\xi}_{t} > 0\}}\right]} \\ 
			&\hspace{-2.5cm}\geq \frac{\mathbb{E}^{(\gamma)}_{(z,x)}\Big[\varphi_1(\mathcal{U}_r)\widehat{U}^{(\gamma)}(\xi_r)\mathbf{1}_{\{\underline{\xi}_{r} > 0\}}\Big]}{\widehat{U}^{(\gamma)}(x)} \int_{0}^{\infty} \Phi_s(z)e^{\gamma z}  \mu_\gamma (\mathrm{d} z)\\ 
			&\hspace{-2.5cm}= \mathbb{E}_{(z,x)}^{(\gamma), \uparrow}[\varphi_1(\mathcal{U}_r)] \int_{0}^{\infty} \Phi_s(z)e^{\gamma z}  \mu_\gamma (\mathrm{d} z).
		\end{split}
	\end{equation}
	Now, we use  the duality relationship, with respect to the Lebesgue measure,  between $\xi$ and $\widehat{\xi}$ (see for instance Lemma 3 in \cite{hirano2001levy}) to get 
	\[
	\begin{split}
		\int_{0}^{\infty} \Phi_s(z)e^{\gamma z} e^{-\gamma z} U^{(\gamma)}(z) \mathrm{d} z &=  \int_{0}^{\infty} \mathbb{E}_{z}^{(e,\gamma)} \left[\varphi_2(\widetilde{W}_{0,s})\varphi_3(\xi_{s})e^{-\gamma \xi_{s}}\mathbf{1}_{\{\underline{\xi}_{s} > 0\}} \right]U^{(\gamma)}(z)\mathrm{d} z\\ 
		&=   \int_{0}^\infty \mathbb{E}_{-z}^{(e,\gamma)} \left[\varphi_2(\widehat{W}_s)U^{(\gamma)}(-\xi_s)\mathbf{1}_{\{\overline{\xi}_{s} < 0\}} \right] \varphi_3(z) e^{-\gamma z}  \mathrm{d} z \\ 
		&=  \int_{0}^\infty \mathbb{E}_{-z}^{(e,\gamma), \downarrow} \left[\varphi_2(\widehat{W}_s)\varphi_3(\xi_0) \right]e^{-\gamma z}  U^{(\gamma)}(z)\mathrm{d} z.
	\end{split}
	\]
	Using this equality in \eqref{eq3_liminf_prop1}, we obtain 
	\begin{equation}\label{eq4_liminf_prop1}
		\liminf_{t \to \infty} \frac{\mathbb{E}^{(\gamma)}_{(z,x)}\left[\varphi_1(\mathcal{U}_r)\varphi_2(\widetilde{W}_{t-s,t})\varphi_3(\xi_t)e^{-\gamma \xi_t}\mathbf{1}_{\{\underline{\xi}_{t} > 0\}}\right] }{\mathbb{E}^{(e,\gamma)}_x\left[e^{-\gamma \xi_t} \mathbf{1}_{\{\underline{\xi}_{t} > 0\}}\right]} \geq  \mathbb{E}_{(z,x)}^{(\gamma), \uparrow}[\varphi_1(\mathcal{U}_r)] \mathbb{E}^{(e,\gamma), \downarrow}_{\mu_\gamma}[\varphi_2(\widehat{W}_s) \varphi_3(\xi_0)].
	\end{equation}
	On the other hand, by  taking $y=x$,  $v=0$  and $h(z)=\varphi_3(z)$ in \eqref{eq_lim_prop1}, we deduce
	\[
	\lim\limits_{t \to \infty} \frac{\mathbb{E}^{(e,\gamma)}_x\left[\varphi_3(\xi_{t}) e^{- \gamma \xi_{t}}\mathbf{1}_{\{\underline{\xi}_{t} > 0\}}\right]}{\mathbb{E}^{(e,\gamma)}_x\left[e^{-\gamma \xi_t}\mathbf{1}_{\{\underline{\xi}_{t} > 0\}}\right]} =  \int_{0}^{\infty} \varphi_3(z)  \mu_\gamma (\mathrm{d} z)=\mathbb{E}^{(e,\gamma), \downarrow}_{\mu_\gamma}[\varphi_3(\xi_0)].
	\]
	Using this last identity and replacing $\varphi_1(\mathcal{U}_r)$ by $1-\varphi(\mathcal{U}_r)$ and  $\varphi_2\equiv 1$  in \eqref{eq4_liminf_prop1},  we get
	\[
	\begin{split}
		\mathbb{E}^{(\gamma), \uparrow}_{(z,x)}\Big[1-\varphi_1(\mathcal{U}_r)\Big] \mathbb{E}^{(e,\gamma), \downarrow}_{\mu_\gamma}&[\varphi_3(\xi_0)]  \leq  
		\liminf_{t \to \infty} \frac{\mathbb{E}^{(\gamma)}_{(z,x)}\Big[\Big(1-\varphi_1(\mathcal{U}_r)\Big)\varphi_3(\xi_t)e^{-\gamma \xi_t}\mathbf{1}_{\{\underline{\xi}_{t} > 0\}}\Big] }{\mathbb{E}^{(e,\gamma)}_x\left[e^{-\gamma \xi_t} \mathbf{1}_{\{\underline{\xi}_{t} > 0\}}\right]} \\ &= \mathbb{E}^{(e,\gamma), \downarrow}_{\mu_\gamma}[\varphi_3(\xi_0)]- \limsup_{t \to \infty} \frac{\mathbb{E}^{(\gamma)}_{(z,x)}\Big[\varphi_1(\mathcal{U}_r)\varphi_3(\xi_t)e^{-\gamma \xi_t}\mathbf{1}_{\{\underline{\xi}_{t} > 0\}}\Big] }{\mathbb{E}^{(e,\gamma)}_x\left[e^{-\gamma \xi_t} \mathbf{1}_{\{\underline{\xi}_{t} > 0\}}\right]}.
	\end{split}
	\]
	Therefore,
	\begin{eqnarray*}
		\limsup_{t \to \infty} \frac{\mathbb{E}^{(\gamma)}_{(z,x)}\Big[\varphi_1(\mathcal{U}_r)\varphi_3(\xi_t)e^{-\gamma \xi_t}\mathbf{1}_{\{\underline{\xi}_{t} > 0\}}\Big] }{\mathbb{E}^{(e,\gamma)}_x\left[e^{-\gamma \xi_t} \mathbf{1}_{\{\underline{\xi}_{t} > 0\}}\right]} \leq 
		\mathbb{E}^{(\gamma), \uparrow}_{(z,x)}[\varphi_1(\mathcal{U}_r)] \mathbb{E}^{(e,\gamma), \downarrow}_{\mu_\gamma}[\varphi_3(\xi_0)].
	\end{eqnarray*}
	In other words, by taking $\varphi_2\equiv 1$ in \eqref{eq4_liminf_prop1}  and the above inequality, we obtain the identity
	\begin{eqnarray*}
		\lim_{t \to \infty} \frac{\mathbb{E}^{(\gamma)}_{(z,x)}\Big[\varphi_1(\mathcal{U}_r)\varphi_3(\xi_t)e^{-\gamma \xi_t}\mathbf{1}_{\{\underline{\xi}_{t} > 0\}}\Big] }{\mathbb{E}^{(e,\gamma)}_x\left[e^{-\gamma \xi_t} \mathbf{1}_{\{\underline{\xi}_{t} > 0\}}\right]} =
		\mathbb{E}^{(\gamma), \uparrow}_{(z,x)}[\varphi_1(\mathcal{U}_r)] \mathbb{E}^{(e,\gamma), \downarrow}_{\mu_\gamma}[\varphi_3(\xi_0)].
	\end{eqnarray*}
	
	Finally we pursue the same strategy as before, that is to say we replace    $\varphi_2(\widetilde{W}_{t-s,t})$  by $1-\varphi_2(\widetilde{W}_{t-s,t})$  in  \eqref{eq4_liminf_prop1} to obtain
	\[
	\begin{split}
		\liminf_{t \to \infty}\frac{\mathbb{E}^{(\gamma)}_{(z,x)}\left[\varphi_1(\mathcal{U}_r)\Big(1-\varphi_2(\widetilde{W}_{t-s,t})\Big)\varphi_3(\xi_t)e^{-\gamma \xi_t}\mathbf{1}_{\{\underline{\xi}_{t} > 0\}}\right] }{\mathbb{E}^{(e,\gamma)}_x\left[e^{-\gamma \xi_t} \mathbf{1}_{\{\underline{\xi}_{t} > 0\}}\right]}&\\
		& \hspace{-3cm} \geq  \mathbb{E}_{(z,x)}^{(\gamma), \uparrow}[\varphi_1(\mathcal{U}_r)] \mathbb{E}^{(e,\gamma), \downarrow}_{\mu_\gamma}\Big[\Big(1-\varphi_2(\widehat{W}_{s})\Big)\varphi_3(\xi_0)\Big].
	\end{split}
	\]
	Then, it follows 
	\[\begin{split}
		\limsup_{t\to\infty}&\frac{\mathbb{E}^{(\gamma)}_{(z,x)}\left[\varphi_1(\mathcal{U}_r)\varphi_2(\widetilde{W}_{t-s,t})\varphi_3(\xi_t)e^{-\gamma \xi_t}\mathbf{1}_{\{\underline{\xi}_{t} > 0\}}\right] }{\mathbb{E}^{(e,\gamma)}_x\left[e^{-\gamma \xi_t} \mathbf{1}_{\{\underline{\xi}_{t} > 0\}}\right]}\\ & \hspace{2cm}\le  \mathbb{E}_{(z,x)}^{(\gamma), \uparrow}[\varphi_1(\mathcal{U}_r)] \mathbb{E}^{(e,\gamma), \downarrow}_{\mu_\gamma}\Big[\varphi_2(\widehat{W}_s) \varphi_3(\xi_0)\Big].
	\end{split}\]
	Finally, putting all pieces together,  we  conclude that 
	\[
	\lim_{t\to\infty}\frac{\mathbb{E}^{(\gamma)}_{(z,x)}\left[\varphi_1(\mathcal{U}_r)\varphi_2(\widetilde{W}_{t-s,t})\varphi_3(\xi_t)e^{-\gamma \xi_t}\mathbf{1}_{\{\underline{\xi}_{t} > 0\}}\right] }{\mathbb{E}^{(e,\gamma)}_x\left[e^{-\gamma \xi_t} \mathbf{1}_{\{\underline{\xi}_{t} > 0\}}\right]}= \mathbb{E}_{(z,x)}^{(\gamma), \uparrow}[\varphi_1(\mathcal{U}_r)] \mathbb{E}^{(e,\gamma), \downarrow}_{\mu_\gamma}\Big[\varphi_2(\widehat{W}_s) \varphi_3(\xi_0)\Big],
	\]
	as expected. 
\end{proof}

The following lemmas are preparatory results for the proof of Theorem \ref{prop_weakly_pareja}.  We first observe from the Wiener-Hopf factorisation  that there exists a non decreasing function $\Psi_0$ satisfying,
$$ \psi_0(\lambda) = \lambda \Psi_0(\lambda), \qquad \text{for} \quad \lambda\geq 0,$$ where $\Psi_0$ is the Laplace exponent of a subordinator and takes the form 
\begin{equation}\label{eq_phi0sub}
\Psi_0(\lambda)=\varrho^2\lambda+\int_{(0,\infty)} (1-e^{-\lambda x})\mu(x,\infty){\rm d}x.
\end{equation}
From  \eqref{betacond},  it follows that $\Psi_0(\lambda)$ is regularly varying at 0 with index $\beta$.

\begin{lemma}\label{lem_cotapsi0}
Let $x, \lambda>0$  and assume that  \eqref{betacond} holds, then
$$\lim\limits_{s \to \infty}\lim\limits_{t\to \infty}e^{-t \Phi_{\xi}(\gamma)} t^{3/2}\int_s^{t-s}\mathbb{E}_{x}^{(e)}\left[\Psi_0(\lambda e^{-\xi_u})\mathbf{1}_{\{\underline{\xi}_t>0\}}\right] {\rm d} u =0.$$
\end{lemma}

\begin{proof}
Let $x>0$ and $\lambda>0$. From  the Markov property, we observe
	\[\begin{split}
		\mathbb{E}_{x}^{(e)}\left[\Psi_0(\lambda e^{-\xi_u})\mathbf{1}_{\{\underline{\xi}_t>0\}}\right]&=\mathbb{E}_{x}^{(e)}\left[\Psi_0(\lambda e^{-\xi_u})\mathbf{1}_{\{\underline{\xi}_u>0\}}\mathbb{P}_{\xi_u}^{(e)}\left(\underline{\xi}_{t-u}>0\right)\right].
	\end{split}\]
 Next we take $x_0 > x$ and from the monotonicity of $z\mapsto\mathbb{P}_{z}^{(e)}(\underline{\xi}_{t-u}>0)$, we obtain
\[
\begin{split}
	\mathbb{E}_{x}^{(e)}\left[\Psi_0(\lambda e^{-\xi_u})\mathbf{1}_{\{\underline{\xi}_t>0\}}\right]
	&\le \mathbb{E}_{x}^{(e)}\left[\Psi_0(\lambda e^{-\xi_u})\mathbf{1}_{\{\underline{\xi}_u>0\}}\mathbb{P}_{\xi_u}^{(e)}\left(\underline{\xi}_{t-u}>0\right)\mathbf{1}_{\{\xi_u>x_0\}}\right]\\
	&\hspace{1cm}+\mathbb{E}_{x}^{(e)}\left[\Psi_0(\lambda e^{-\xi_u})\mathbf{1}_{\{\underline{\xi}_u>0\}}\mathbf{1}_{\{\xi_u\le x_0\}}\right]\mathbb{P}_{x_0+x}^{(e)}\left(\underline{\xi}_{t-u}>0\right).
\end{split}
\]
Now using the asymptotic behaviour given in \eqref{eq_weakly_cota_hirano} and the Escheer transform \eqref{eq_medida_Essher}, for $t$ large enough, we have 
\begin{equation}\label{eq:cota}
\begin{split}
	\mathbb{E}_{x}^{(e)}\Big[\Psi_0&(\lambda e^{-\xi_u})\mathbf{1}_{\{\underline{\xi}_t>0\}}\big]\\ &  \hspace{1cm} \le C_{\gamma} \mathbb{E}_{x}^{(e)}\left[\Psi_0(\lambda e^{-\xi_u})\mathbf{1}_{\{\underline{\xi}_u>0\}}\mathbf{1}_{\{\xi_u>x_0\}}e^{\gamma \xi_u} \widehat{U}^{(\gamma)} (\xi_u)\right](t-u)^{-3/2}  e^{\Phi_\xi(\gamma)(t-u)}\\
	&\hspace{1.5cm}+C_{\gamma, x+x_0}\mathbb{E}_{x}^{(e)}\left[\Psi_0(\lambda e^{-\xi_u})\mathbf{1}_{\{\underline{\xi}_u>0\}}\mathbf{1}_{\{\xi_u\le x_0\}}\right]\left(t-u\right)^{-3/2}e^{\Phi_\xi(\gamma)(t-u)}\\
	&\hspace{1cm} \le  C_{\gamma} \mathbb{E}_{x}^{(e, \gamma)}\left[\Psi_0(\lambda e^{-\xi_u})\mathbf{1}_{\{\underline{\xi}_u>0\}}\mathbf{1}_{\{\xi_u>x_0\}}\widehat{U}^{(\gamma)} (\xi_u) \right]\left(t-u\right)^{-3/2}  e^{\Phi_\xi(\gamma)t}\\
	&\hspace{1.5cm}+C_{\gamma, x+x_0}\mathbb{E}_{x}^{(e)}\left[\Psi_0(\lambda e^{-\xi_u})\mathbf{1}_{\{\underline{\xi}_u>0\}}\mathbf{1}_{\{\xi_u\le x_0\}}\right]\left(t-u\right)^{-3/2}  e^{\Phi_\xi(\gamma)(t-u)},
\end{split}
\end{equation}
where $C_{\gamma}$ and $C_{\gamma, x+x_0}$ are strictly positive constants. 

First, we deal with the first expectation in the right-hand side of the previous inequality. Recalling that $\Phi_\xi^{\prime\prime}(\gamma)<\infty$, we get  from Corollary 5.3 in \cite{kyprianou2014fluctuations}   that 
\[
y^{-1}\widehat{U}^{(\gamma)}(y)\to \frac{1}{\widehat{\mathbb{E}}^{(e,\gamma)}[H_1]}, \qquad \textrm{as} \quad y\to \infty.
\]
Furthermore, since $\widehat{U}^{(\gamma)}$ is increasing then the map $y\mapsto e^{-\frac{\varsigma}{2} y}\widehat{U}^{(\gamma)}(y)$ is bounded  for any $\varsigma \in (0,\beta)$ and from \eqref{betacond}, we also deduce  that the map $y\mapsto e^{-\frac{\varsigma}{2} y}\ell(\lambda e^{-y})$ is also bounded. 
%By  \eqref{rem:boundonpsi}, we have that $\beta>0$ implies $$\lim_{\lambda \to 0}  \lambda^{-\beta}\Psi_0(\lambda)<\infty.$$
With these observations in mind and, it follows, for $u$ large enough, that
\[\begin{split}
	\mathbb{E}_{x}^{(e, \gamma)}\left[\Psi_0(\lambda e^{-\xi_u})\mathbf{1}_{\{\underline{\xi}_u>0\}}\mathbf{1}_{\{\xi_u>x_0\}}\widehat{U}^{(\gamma)} (\xi_u) \right]
	%& \leq 	C_1\mathbb{E}_{x}^{(e, \gamma)}\left[\Psi_0(\lambda e^{-\xi_u})\mathbf{1}_{\{\underline{\xi}_u>0\}} \mathbf{1}_{\{\xi_u>x_0\}} e^{\varsigma \xi_u}\right] \\ &
	\leq C_\lambda  \mathbb{E}_{x}^{(e, \gamma)}\left[e^{-(\beta-\frac{\varsigma}{2})\xi_u}\mathbf{1}_{\{\underline{\xi}_u>0\}} \right],
\end{split}\]
where $C_\lambda$ is a strictly positive constants. According to Lemma 1 in \cite{hirano2001levy}, we have that for $u$ sufficiently large there exists $C_{\lambda, \beta, x}$ such that 
%$$ \mathbb{E}^{(e,\gamma)}_{x}[e^{-b \xi_{u} }\mathbf{1}_{\{\underline{\xi}_{u}>0\}}]\big] \leq C_3 e^{b x}\widehat{U}^{(\gamma)}(x)  u^{-3/2}, \qquad x>0.$$
%Hence, 
\[\begin{split}
	\mathbb{E}_{x}^{(e, \gamma)}\left[\Psi_0(\lambda e^{-\xi_u})\mathbf{1}_{\{\underline{\xi}_u>0\}}\mathbf{1}_{\{\xi_u>x_0\}}\widehat{U}^{(\gamma)} (\xi_u) \right]& \leq C_{\lambda, \beta, x} u^{-3/2}.
\end{split}\]
  For the second expectation in \eqref{eq:cota},  we use the monotonicity of  $\Psi_0$ to get
\[
\begin{split}
	\mathbb{E}_{x}^{(e)}\left[\Psi_0(\lambda e^{-\xi_u})\mathbf{1}_{\{\underline{\xi}_u>0\}}\mathbf{1}_{\{\xi_u\le x_0\}}\right] & \leq 
\Psi_0(\lambda)\mathbb{P}_{x}^{(e)}\left(\underline{\xi}_u>0\right) \leq \widehat{C}_{\gamma, x, \lambda} u^{-3/2}e^{\Phi_\xi(\gamma)u},
\end{split}
\]
where $\widehat{C}_{\gamma, x, \lambda}$ is a positive constant.  Putting all pieces together in \eqref{eq:cota}, we deduce, for $t$ large enough, that
 \[	\mathbb{E}_{x}^{(e)}\Big[\Psi_0(\lambda e^{-\xi_u})\mathbf{1}_{\{\underline{\xi}_t>0\}}\big] \leq C_{\lambda, \beta, x, \gamma} u^{-3/2} (t-u)^{-3/2} e^{\Phi_{\xi}(\gamma) t},\]
 where $C_{\lambda, \beta, x, \gamma}>0$. Finally, observe that for $t$ large enough
\[\begin{split}
	e^{-t \Phi_{\xi}(\gamma)} t^{3/2} \int_s^{t-s}\mathbb{E}_{x}^{(e)}\left[\Psi_0(\lambda e^{-\xi_u})\mathbf{1}_{\{\underline{\xi}_t>0\}}\right] {\rm d} u  & \leq C_{\lambda, \beta, x, \gamma} t^{3/2} \int_{s}^{t-s}  (t-u)^{-3/2}u^{-3/2} {\rm d} u \\ & \leq 2 C_{\lambda, \beta, x, \gamma} t^{3/2}   \left(\frac{t}{2}\right)^{-3/2} \int_{s}^{\infty} u^{-3/2} \mathrm{d} u \\ & \leq  2 C_{\lambda, \beta, x, \gamma} s^{-1/2}.
\end{split}\]
The result now follows by taking $t\to \infty$ and then $s\to \infty$.

%	Now using the asymptotic behaviour given in \eqref{eq_weakly_cota_hirano} and the Escheer transform \eqref{eq_medida_Essher}, for $t\ge t_0$ , we have 
%	\[
%	\begin{split}
%		\mathbb{E}_{x}^{(e)}\Big[\Psi_0(\lambda e^{-\xi_u})\mathbf{1}_{\{\underline{\xi}_t>0\}}\big] &\le C_{\gamma} \mathbb{E}_{x}^{(e)}\left[\Psi_0(\lambda e^{-\xi_u})\mathbf{1}_{\{\underline{\xi}_u>0\}}e^{\gamma \xi_u} \widehat{U}^{(\gamma)} (\xi_u)\right](t-u)^{-3/2}e^{\Phi_\xi(\gamma)(t-u)}\\
%		&\leq C_{\gamma}\mathbb{E}_{x}^{(e, \gamma)}\left[\Psi_0(\lambda e^{-\xi_u})\mathbf{1}_{\{\underline{\xi}_u>0\}} \widehat{U}^{(\gamma)} (\xi_u)\right](t-u)^{-3/2}e^{\Phi_\xi(\gamma)t}.
%	\end{split}
%	\]
%	where $C_{\gamma}$ and $C_{\gamma, x}$ are positive constants that only depend on $\gamma$ and $x$.
%	Finally,  recalling  $\phi_\xi^{\prime\prime}(\gamma)<\infty$, we get  from Corollary 5.3 in \cite{kyprianou2014fluctuations}   that 
%	\[
%	y^{-1}\widehat{U}^{(\gamma)}(y)\to \frac{1}{\widehat{\mathbb{E}}^{(e,\gamma)}[H_1]}, \qquad \textrm{as} \quad y\to \infty.
%	\]
%	Moreover,  $\widehat{U}^{(\gamma)}$ is increasing then the map $y\mapsto e^{-\theta y}\widehat{U}^{(\gamma)}(y)$ is a bounded function for any $\theta \in (0,\beta)$. According to Proposition 3.1 in \cite{li2018asymptotic}, we have that for $u$ sufficiently large there exists a positive constant $C_2$ such that 
%	$$ \mathbb{E}^{(e,\gamma)}_{x}[e^{-\lambda \xi_{u} }\mathbf{1}_{\{\underline{\xi}_{u}>0\}}]\big] \leq C_2 \widehat{U}^{(\gamma)}(x)  u^{-3/2}, \qquad x>0.$$
	
\end{proof}

\begin{lemma}\label{lem_cotasv}
		Let $z, x >0$ and assume that \eqref{betacond} holds, then
\[\begin{split}
\lim\limits_{s \to \infty}	\lim\limits_{t\to \infty} & t^{3/2}e^{-t \Phi_{\xi}(\gamma)}\mathbb{E}_{(z,x)}\Bigg[\bigg|\exp\left\{-Z_s e^{-\xi_s}v_t(s,\lambda, \xi)\right\}\bigg. \Bigg. \\ &\Bigg.\bigg. \hspace{4cm}- \exp\left\{-Z_s e^{-\xi_s}v_t(t-s,\lambda, \xi)\right\}\bigg| \mathbf{1}_{\{\underline{\xi}_t>0\}}  \Bigg] =0.
\end{split}\]

\end{lemma}

\begin{proof}
Fix $z,x>0$ and take $t\geq 2s$. We begin by observing that since $f(y)=e^{-y}$, $y\geq 0$, it is  Lipschitz  and hence there exists a positive constant $C_1$ such that
\[\begin{split}
	\mathbb{E}_{(z,x)}&\Bigg[\bigg|\exp\left\{-Z_s e^{-\xi_s}v_t(s,\lambda, \xi)\right\}- \exp\left\{-Z_s e^{-\xi_s}v_t(t-s,\lambda, \xi)\right\}\bigg| \mathbf{1}_{\{\underline{\xi}_t>0\}}\Bigg]\\ & \hspace{1.5cm } \leq C_1	\mathbb{E}_{(z,x)}\left[Z_s e^{-\xi_s}\big|v_t(s,\lambda, \xi)-v_t(t-s,\lambda, \xi)\big| \mathbf{1}_{\{\underline{\xi}_t>0\}} \right] \\ &  \hspace{1.5cm} =  C_1 z^{-1}	\mathbb{E}_{x}^{(e)}\left[\big|v_t(s,\lambda, \xi)-v_t(t-s,\lambda, \xi)\big| \mathbf{1}_{\{\underline{\xi}_t>0\}}\right],
\end{split}\]
where in the last identity we have conditioned on the environment and used \eqref{martingquenched}. Since $\psi_0$ is positive,   from \eqref{eq_BDE} we have  that $s\mapsto v_t(s, \lambda, \xi)$ is an increasing function. This together with the facts that $\psi_0$ is a non-decreasing function and  $v_t(t, \lambda,\xi) = \lambda$, we deduce  $$ \psi_0\big(v_t(u, \lambda, \xi) e^{-\xi_u}\big)\leq \psi_0(\lambda e^{-\xi_u}), \qquad \text{for} \quad u\leq t.$$
Hence, we obtain 
\[\begin{split}
v_t(s,\lambda, \xi)-v_t(t-s,\lambda, \xi) &= \int_{s}^{t-s} e^{\xi_u} \psi_0\big(v_t(u, \lambda, \xi) e^{-\xi_u}\big)\mathrm{d} u\\ & \leq  \int_{s}^{t-s}  e^{\xi_u} \psi_0(\lambda e^{-\xi_u} ) \mathrm{d} u= \int_{s}^{t-s}  \lambda \Psi_0(\lambda e^{-\xi_u} )\mathrm{d} u.
\end{split}\]
In other words, we have deduced
\[
\mathbb{E}_{x}^{(e)} \left[\big|v_t(s,\lambda, \xi)-v_t(t-s,\lambda, \xi)\big| \mathbf{1}_{\{\underline{\xi}_t>0\}}  \right] \leq  \lambda \int_{s}^{t-s} \mathbb{E}_{x}^{(e)}\left[\Psi_0(\lambda e^{-\xi_u})\mathbf{1}_{\{\underline{\xi}_t>0\}}\right] \mathrm{d} u.
\]
Appealing to Lemma \ref{lem_cotapsi0}, we conclude that
\[\begin{split}
	\lim\limits_{s \to \infty}	\lim\limits_{t\to \infty} & t^{3/2}e^{-t \Phi_{\xi}(\gamma)}\mathbb{E}_{(z,x)}\Bigg[\bigg|\exp\left\{-Z_s e^{-\xi_s}v_t(s,\lambda, \xi)\right\}\bigg. \Bigg. \\ &\Bigg.\bigg. \hspace{4cm}- \exp\left\{-Z_s e^{-\xi_s}v_t(t-s,\lambda, \xi)\right\}\bigg|\mathbf{1}_{\{\underline{\xi}_t>0\}} \Bigg] \\ & \leq C_1 z^{-1} \lambda	\lim\limits_{s \to \infty}	\lim\limits_{t\to \infty}t^{3/2}e^{-t \Phi_{\xi}(\gamma)}  \int_s^{t-s}\mathbb{E}_{x}^{(e)}\left[\Psi_0(\lambda e^{-\xi_u})\mathbf{1}_{\{\underline{\xi}_t>0\}}\right] {\rm d} u = 0,
\end{split}\]
as required.
\end{proof}

The following lemma states that, with respect to the measure $\mathbb{P}_{(z,x)}^{(\gamma),\uparrow}$ with $z,x>0$, the reweighted  process $(Z_te^{-\xi_t}, t\geq 0)$ is a martingale that  converges towards a strictly positive r.v. under $\mathbb{P}_{(z,x)}^{(\gamma),\uparrow}$. This is another preparatory lemma for the proof of Theorem  \ref{prop_weakly_pareja} below.

\begin{lemma}\label{lem_mtg}
	Let $z,x >0$ and assume that  \eqref{betacond} holds. Then the  process $(Z_te^{-\xi_t},t\geq 0)$ is a martingale with respect to $(\mathcal{F}_t)_{t\geq 0}$ under $\mathbb{P}_{(z,x)}^{(\gamma),\uparrow}$. Moreover, as $t\to\infty$
	$$  Z_{t}e^{-\xi_{t}} \longrightarrow \mathcal{U}_\infty, \quad \quad \mathbb{P}_{(z,x)}^{(\gamma),\uparrow}-\text{a.s.},$$ 
	where the random variable $\mathcal{U}_\infty$ is finite and satisfies
	$$  \mathbb{P}_{(z,x)}^{(\gamma),\uparrow}\big(\mathcal{U}_\infty>0\big)>0.$$
\end{lemma}

In order to prove the above result, we require the following Lemma which is Proposition 3.4 in Bansaye et al. \cite{bansaye2021extinction}.

\begin{lemma}[Proposition 3.4 in \cite{bansaye2021extinction}] \label{lem_JCProba} Let $z,x>0$  and assume that the environment $\xi$ is critical under $ \mathbb{P}_{(z,x)}$  and that  \eqref{eq_xlog2x} is fulfilled,  then 
	$$ \lim\limits_{t\to \infty} \mathbb{P}^{\uparrow}_{(z,x)}\big(Z_t>0\big) >0.$$		
\end{lemma}
We recall that \eqref{betacond} implies the $x\log^2(x)$-moment condition \eqref{eq_xlog2x}. 
\begin{proof} [Proof of Lemma \ref{lem_mtg}]
	From Proposition 1.1 in \cite{bansaye2021extinction}, which we may apply here with respect to the measure $\mathbb{P}_{(z,x)}^{(\gamma)}$, we have that the process $(Z_te^{-\xi_t},t\geq 0)$ is a quenched martingale with respect to the environment.
	% under $\mathbb{P}_{(z,x)}^{(\gamma)}$
	We assume that $s\leq t$ and take $A\in \mathcal{F}_s$. In order to deduce the first claim of this lemma, we first show
	\[\begin{split}
		\mathbb{E}_{(z,x)}^{(\gamma)}\Big[Z_te^{-\xi_t}\mathbf{1}_A\widehat{U}^{(\gamma)}(\xi_t)\mathbf{1}_{\{\underline{\xi}_t>0\}}\Big] &= 		\mathbb{E}_{(z,x)}^{(\gamma)}\Big[Z_se^{-\xi_s}\mathbf{1}_A\widehat{U}^{(\gamma)}(\xi_t)\mathbf{1}_{\{\underline{\xi}_t>0\}}\Big].
	\end{split}\]
	First, conditioning on the environment, we deduce that 
	\[\begin{split}
		\mathbb{E}_{(z,x)}^{(\gamma)}\Big[Z_te^{-\xi_t}\mathbf{1}_A\widehat{U}^{(\gamma)}(\xi_t)\mathbf{1}_{\{\underline{\xi}_t>0\}}\Big] &= \mathbb{E}_{(z,x)}^{(\gamma)}\Big[\mathbb{E}_{(z,x)}^{(\gamma)}[Z_te^{-\xi_t}\mathbf{1}_A| \xi]\widehat{U}^{(\gamma)}(\xi_t)\mathbf{1}_{\{\underline{\xi}_t>0\}}\Big]\\ &= \mathbb{E}_{(z,x)}^{(\gamma)}\Big[\mathbb{E}_{(z,x)}^{(\gamma)}[Z_se^{-\xi_s}\mathbf{1}_A|  \xi]\widehat{U}^{(\gamma)}(\xi_t)\mathbf{1}_{\{\underline{\xi}_t>0\}}\Big].
	\end{split}\]
	We can see that the random variable $\mathbb{E}_{(z,x)}^{(\gamma)}[Z_se^{-\xi_s}\mathbf{1}_A|  \xi]$ is $\mathcal{F}_s$-measurable. Thus conditioning on $\mathcal{F}_s$, we have 
	\[\begin{split}
		\mathbb{E}_{(z,x)}^{(\gamma)}\Big[Z_te^{-\xi_t}\mathbf{1}_A\widehat{U}^{(\gamma)}(\xi_t)\mathbf{1}_{\{\underline{\xi}_t>0\}}\Big] &= \mathbb{E}_{(z,x)}^{(\gamma)}\Big[\mathbb{E}_{(z,x)}^{(\gamma)}[Z_se^{-\xi_s}\mathbf{1}_A |  \xi] \mathbb{E}_{(z,x)}^{(\gamma)}[\widehat{U}^{(\gamma)}(\xi_t)\mathbf{1}_{\{\underline{\xi}_t>0\}}| \mathcal{F}_s]\Big].
	\end{split}\]
	Further, by Lemma 3.1 in \cite{bansaye2021extinction}, which we can apply here under the measure $\mathbb{P}_{(z,x)}^{(\gamma)}$,
	%since it satisfies the assumptions  \eqref{eq_xlog2x} assumed in that result, 
	the process $(\widehat{U}^{(\gamma)}(\xi_t)\mathbf{1}_{\{\underline{\xi}_t>0\}}, t\geq 0)$ is a martingale with respect to $(\mathcal{F}_t)_{t\geq 0}$ under $\mathbb{P}_{(z,x)}^{(\gamma)}$. Hence
	\[\begin{split}
		\mathbb{E}_{(z,x)}^{(\gamma)}\Big[Z_te^{-\xi_t}\mathbf{1}_A\widehat{U}^{(\gamma)}(\xi_t)\mathbf{1}_{\{\underline{\xi}_t>0\}}\Big] &=  \mathbb{E}_{(z,x)}^{(\gamma)}\Big[\mathbb{E}_{(z,x)}^{(\gamma)}[Z_se^{-\xi_s}\mathbf{1}_A |  \xi] \widehat{U}^{(\gamma)}(\xi_s)\mathbf{1}_{\{\underline{\xi}_s>0\}}\Big] \\ &= \mathbb{E}_{(z,x)}^{(\gamma)}\Big[Z_se^{-\xi_s}\mathbf{1}_A \widehat{U}^{(\gamma)}(\xi_s)\mathbf{1}_{\{\underline{\xi}_s>0\}}\Big]. 
	\end{split}\]
	Therefore, by definition of the measure $\mathbb{P}_{(z,x)}^{(\gamma),\uparrow}$ we see 
	\[\begin{split}
		\mathbb{E}_{(z,x)}^{(\gamma),\uparrow}\big[Z_te^{-\xi_t}\mathbf{1}_A\big] &= \frac{1}{\widehat{U}(x)} \mathbb{E}_{(z,x)}^{(\gamma)}\Big[Z_te^{-\xi_t}\mathbf{1}_A\widehat{U}^{(\gamma)}(\xi_t)\mathbf{1}_{\{\underline{\xi}_t>0\}}\Big] \\& = \frac{1}{\widehat{U}(x)} \mathbb{E}_{(z,x)}^{(\gamma)}\Big[Z_se^{-\xi_s}\mathbf{1}_A \widehat{U}^{(\gamma)}(\xi_s)\mathbf{1}_{\{\underline{\xi}_s>0\}}\Big] = \mathbb{E}_{(z,x)}^{(\gamma),\uparrow}\big[Z_se^{-\xi_s}\mathbf{1}_A\big],
	\end{split}\]
	which allows us to conclude that the process $(Z_te^{-\xi_t},t\geq 0)$ is a martingale with respect to $(\mathcal{F}_t)_{t\geq 0}$ under $\mathbb{P}_{(z,x)}^{(\gamma),\uparrow}$. Moreover, by Doob's convergence theorem, there is a non-negative finite r.v. $\mathcal{U}_\infty$ such that as $t\to \infty$
	$$  Z_{t}e^{-\xi_{t}} \longrightarrow \mathcal{U}_\infty, \quad \quad \mathbb{P}_{(z,x)}^{(\gamma),\uparrow}-\text{a.s.}$$ 
	Next, by Dominated Convergence Theorem we have
	\[\mathbb{P}_{(z,x)}^{(\gamma),\uparrow}\big(\mathcal{U}_\infty>0\big) = \lim\limits_{t\to\infty}\mathbb{P}_{(z,x)}^{(\gamma),\uparrow}\big(Z_te^{-\xi_t}>0\big).\]
	The proof is thus completed as soon as we can show 
	\begin{equation}\label{eq_pgamma}
		\lim\limits_{t\to\infty}\mathbb{P}_{(z,x)}^{(\gamma),\uparrow}\big(Z_te^{-\xi_t}>0\big)>0.
	\end{equation}
	In order to do so, we first observe that the following identity holds
	\[\mathbb{P}_{(z,x)}^{(\gamma),\uparrow}\big(Z_te^{-\xi_t}=0\big) %=\mathbb{E}_{x}^{(e,\gamma),\uparrow}\big[e^{-zv_t(0,\infty, \xi -x)}\big]
	=\mathbb{P}_{(z,x)}^{(\gamma),\uparrow}\big(Z_t=0\big),\]
	then  by noting that under $\mathbb{P}^{(\gamma)}_{(z,x)}$ the L\'evy process $\xi$ oscillates (since $\Phi_{\xi}'(\gamma)=0$), we can apply Lemma \ref{lem_JCProba}  to deduce \eqref{eq_pgamma}. 
	
\end{proof}

With Proposition \ref{prop_weakly_primera} and Lemmas  \ref{lem_cotasv} and \ref{lem_mtg}  in hand, we may now proceed to prove  Theorem \ref{prop_weakly_pareja} following similar ideas as those used in Lemma 3.4 in \cite{afanasyev2012} although we might  consider that the  continuous setting leads to significant changes since an extension of Proposition \ref{prop_weakly_primera} seems difficult to be deduced unlike in the discrete case (see Theorem 2.7 in \cite{afanasyev2012}). Indeed, it seems that such extension  will depend on a much deeper analysis on the asymptotic behaviour for bridges of L\'evy processes and their conditioned version.

\begin{proof}[Proof of Theorem \ref{prop_weakly_pareja}]
	Fix $x,z>0$ and recall that the process $(\mathcal{U}_s, s\geq 0)$ is defined as $	\mathcal{U}_{s} := Z_{s}e^{-\xi_{s}}$. For any $\lambda\geq 0$, we shall prove the convergence of  the following Laplace transform as $t\to\infty$, $$\mathbb{E}_{(z,x)}\left[\exp\{-\lambda Z_te^{-\xi_t}\}\ \big|\big. \  \underline{\xi}_t >0\right].$$  
	First  we rewrite the latter expression in a form which allows to use  Proposition \ref{prop_weakly_primera} and Lemma \ref{lem_cotasv}. We begin by recalling from \eqref{eq_Laplace}, that for any $\lambda\geq 0$ and $t\geq s\geq 0$  we have
	\begin{equation*}\label{eq_quenchedlaw}
		\mathbb{E}_{(z,x)}\left[\exp\{-\lambda Z_te^{-\xi_t}\}\ \big|\big.\  \xi, \mathcal{F}^{(b)}_s\right] = \exp\big\{- Z_se^{-\xi_s}v_t(s,\lambda,\xi)\big\}.
	\end{equation*} 
Thus
\[\begin{split}
	\mathbb{E}_{(z,x)}&\left[\exp\{-\lambda Z_te^{-\xi_t}\}\mathbf{1}_{\{\underline{\xi}_t>0\}}\right]= \mathbb{E}_{(z,x)}\Big[\mathbb{E}_{(z,x)}\left[\exp\{-\lambda Z_te^{-\xi_t}\}\ \big|\big.\  \xi, \mathcal{F}^{(b)}_s\right] \mathbf{1}_{\{\underline{\xi}_t>0\}}\Big]\\ &\hspace{0.3cm} = \mathbb{E}_{(z,x)} \Big[ \exp\big\{- Z_se^{-\xi_s}v_t(s,\lambda,\xi)\big\}\mathbf{1}_{\{\underline{\xi}_t>0\}}\Big]\\ &\hspace{0.3cm}  =  \mathbb{E}_{(z,x)} \Big[ \exp\big\{- Z_se^{-\xi_s}v_t(t-s,\lambda,\xi)\big\}\mathbf{1}_{\{\underline{\xi}_t>0\}}\Big] \\ &\hspace{1cm}+  \mathbb{E}_{(z,x)} \Big[\Big( \exp\big\{- Z_se^{-\xi_s}v_t(s,\lambda,\xi)\big\}- \exp\big\{- Z_se^{-\xi_s}v_t(t-s,\lambda,\xi)\big\}\Big)\mathbf{1}_{\{\underline{\xi}_t>0\}}\Big].
\end{split}\]
	Now, using the same notation as in Proposition \ref{prop_weakly_primera}, we note that for any $s\leq t$, 
	\begin{equation}\label{eq_vaphi}
	  \exp\left\{-{Z_{s}e^{-\xi_{s}}} v_{t}(t-s,\lambda,\xi)\right\} = \varphi\big(\mathcal{U}_{s} ,\widetilde{W}_{t-s,t}, \xi_t\big),
	\end{equation}
	where  $(\widehat W_s(\lambda), s\geq 0)$ and $(\widetilde{W}_{t-s,t}, s\leq t)$ are defined by
	\[\begin{split}
	\widehat{W}_{s}(\lambda):= \exp\left\{-v_s(0,\lambda,\widehat{\xi})\right\},	\quad \quad  \quad \widetilde{W}_{t-s, t}:=\exp\left\{-v_t(t-s,\lambda,\xi)\right\},
	\end{split}\]
and $\varphi$  is the following bounded and continuous function 
		$$\varphi(\textbf{u},w,y) := w^{\textbf{u}}, \qquad  0\leq w\leq 1, \quad \textbf{u}\geq 0, \quad y\in \mathbb{R}.$$ 
	Hence, appealing to Proposition \ref{prop_weakly_primera}, Lemma \ref{lem_cotasv} and 
	\eqref{eq_weakly_cota_hirano},  for $z,x>0$, we see
	\[\begin{split}
		\lim\limits_{t\to \infty}& \mathbb{E}_{(z,x)}\left[\exp\{-\lambda Z_t e^{-\xi_t}\}\ \Big|\Big. \  \underline{\xi}_t >0\right] = \lim\limits_{s\to \infty} \lim\limits_{t\to \infty}	\mathbb{E}_{(z,x)}\left[\varphi(\mathcal{U}_{s},\widetilde{W}_{t-s,t}, \xi_t)\ \Big|\Big.\   \underline{\xi}_t > 0\right]  \\ & \hspace{0.1cm} +   \lim\limits_{s\to \infty}\lim\limits_{t\to \infty}\mathbb{E}_{(z,x)}\left[\Big|\exp\left\{-Z_s e^{-\xi_s}v_t(s,\lambda, \xi)\right\} - \exp\left\{-Z_s e^{-\xi_s}v_t(t-s,\lambda, \xi)\right\}\Big|\, \bigg|\Big. \ \underline{\xi}_t>0  \right]  \\&\hspace{0.3cm}= 	\lim\limits_{s\to \infty}\lim\limits_{t\to \infty}\frac{	\mathbb{E}_{(z,x)}^{(\gamma)}\left[\varphi(\mathcal{U}_{s},\widetilde{W}_{t-s,t}, \xi_t)e^{-\gamma \xi_t}\mathbf{1}_{\{\underline{\xi}_t>0\}}\right] }{\mathbb{E}^{(e, \gamma)}_x\left[e^{-\gamma \xi_t}\mathbf{1}_{\{\underline{\xi}_t>0\}}\right]} = 	\lim\limits_{s\to \infty}\Upsilon_{z,x}(\lambda,s),
	\end{split}\]
	where
	\begin{equation*}
		\Upsilon_{z,x}(\lambda,s):=	 \iiint \varphi(\textbf{u},w,y)  \mathbb{P}^{(\gamma),\uparrow}_{(z,x)}\big(\mathcal{U}_s \in \mathrm{d} \textbf{u}\big) \mathbb{P}^{(e,\gamma),\downarrow}_{-y}\big(W_s(\lambda) \in \mathrm{d} w\big)\mu_\gamma(\mathrm{d}  y).
	\end{equation*}
	On the other hand, from Lemma \ref{lem_mtg},  we recall that, under  $\mathbb{P}_{(z,x)}^{(\gamma),\uparrow}$, the process $(\mathcal{U}_s, s\geq 0)$ is a non-negative martingale with respect to $(\mathcal{F}_t)_{t\geq 0}$ that converges towards  the non-negative and finite r.v.  $\mathcal{U}_\infty$. Next,  we observe from Proposition 2.3 in \cite{he2018continuous}  that the mapping $s\mapsto v_s(0,\lambda,\widehat \xi)$ is decreasing implying that $s\mapsto \widehat W_s(\lambda)$ is increasing $\mathbb{P}_{-y}^{(e,\gamma),\downarrow}$-a.s., for $y>0$. Further, since $v_{s}(0,\lambda, \widehat \xi)\leq \lambda$,  the process  $(\widehat W_{s}(\lambda), s\geq 0)$ is bounded below, i.e. for any $\lambda\geq 0$,
$$ 0 < e^{-\lambda}\leq \widehat W_{s}(\lambda) \leq 1.$$
Therefore it follows that, for any $\lambda\geq 0$ and $y>0$,
\begin{equation}\label{eq_convW}
	\widehat W_{s}(\lambda) \xrightarrow[s\to\infty]{} \widehat W_\infty(\lambda), \quad \quad \mathbb{P}_{-y}^{(e,\gamma),\downarrow}-\text{a.s.},
\end{equation}
where  $\widehat W_\infty(\lambda)$ is a strictly positive r.v.
	The above observations together with the dominated convergence theorem  imply that  
	$$ 	\lim\limits_{s\to \infty}\Upsilon_{z,x}(\lambda,s) = \iiint \varphi(\textbf{u},w,y)  \mathbb{P}^{(\gamma),\uparrow}_{(z,x)}\big(\mathcal{U}_\infty \in \mathrm{d} \textbf{u}\big) \mathbb{P}^{(e,\gamma),\downarrow}_{-y}\big(\widehat W_\infty(\lambda) \in \mathrm{d} w\big)\mu_\gamma(\mathrm{d}  y):= \Upsilon_{z,x}(\lambda).$$
In other words  $ \mathcal{U}_t=Z_t e^{-\xi_t}$ converges weakly, under $\mathbb{P}_{(z,x)}\big(\cdot \ | \ \underline{\xi}_t>0 \big)$,  towards some positive and finite r.v. that we denote by $Q$ and whose Laplace transform is given by $\Upsilon_{z,x}$.

 Next, we observe that  the probability of the event $\{Q>0\}$ is strictly positive. The latter is equivalent to show that  $\Upsilon_{z,x}(\lambda)<1$ for all $\lambda>0$. In other words, from the  definition of $\varphi(\textbf{u},w,y)$, it is enough to show
	$$ \mathbb{P}_{(z,x)}^{(\gamma),\uparrow}\big(\mathcal{U}_\infty>0\big)>0 \quad \quad  \text{and}\quad \quad  \mathbb{P}_{-y}^{(e,\gamma),\downarrow}\big(\widehat W_\infty(\lambda)<1\big)=1, \quad \text{for all}\quad \lambda>0.$$ 
	The first claim has been proved in Lemma \ref{lem_mtg}. For the second claim, we observe that for any $\lambda>0$, 
\[ \mathbb{P}_{-y}^{(e,\gamma),\downarrow}\big(\widehat W_\infty(\lambda)<1\big) = \mathbb{P}_{-y}^{(e,\gamma), \downarrow}\big(v_\infty(0, \lambda, \widehat{\xi})>0\big). \]
By the proof of Proposition 3.4. in \cite{bansaye2021extinction}, we have 
\[v_\infty(0, \lambda, \xi) \geq \lambda \exp\left\{-\int_{0}^{\infty} \Psi_0(\lambda e^{-\xi_u}){\rm d} u\right\},\]
and moreover, from the same reference and under assumption \eqref{betacond}, it follows 
\[\mathbb{E}^{(e, \gamma),  \uparrow}_y \left[\int_{0}^{\infty} \Psi_0(\lambda e^{-\xi_u}){\rm d} u\right] < \infty,\]
which implies that 
$$\mathbb{P}_{-y}^{(e,\gamma), \downarrow }\big(v_\infty(0, \lambda, \widehat \xi)>0\big)=1, \quad \text{for all} \quad \lambda \geq 0.$$
In other words, the probability of the event $\{Q>0\}$ is strictly positive, which implies
	$$\lim\limits_{t\to \infty}\mathbb{P}_{(z,x)}\Big(Z_t e^{-\xi_t}>0 \ \Big|\Big. \ \underline{\xi}_t>0\Big)>0.$$		
	This completes the proof. 
\end{proof}

\subsection{Proof of Theorem \ref{teo_debil}}

The proof of this theorem follows a similar strategy as the proof of  Theorem 1.2 in  Bansaye et al. \cite{bansaye2021extinction} for the critical regime where assumption that $\ell(\lambda)>C$, for $C>0$, and the asymptotic behaviour of exponential functionals of L\'evy processes are crucial. We also recall that   $Z$ is in the weakly subcritical regime.

For simplicity of exposition, we split  the proof of  Theorem \ref{teo_debil} into two  lemmas. The first Lemma  is a direct consequence of Theorem \ref{prop_weakly_pareja}.
\begin{lemma}\label{lem_weakly_2}
	Suppose that \eqref{betacond} holds.  Then for  any $z, x >0$ we have, as $t \to \infty$
	\begin{eqnarray*}
		\mathbb{P}_{(z,x)}\Big(Z_t >0,\  \underline{\xi}_t>0\Big) &\sim& \mathfrak{b}(z,x)  \mathbb{P}^{(e)}_x\left(\underline{\xi}_t > 0\right) \\ &\sim&  \mathfrak{b}(z,x)\frac{A_\gamma}{ \gamma\kappa^{(\gamma)}(0,\gamma)} e^{\gamma x} \widehat{U}^{(\gamma)}(x) t^{-3/2}e^{\Phi_{\xi}(\gamma)t},
	\end{eqnarray*}
	where the constant $A_\gamma$ is defined in \eqref{eq_weakly_constante}.
\end{lemma}

\begin{proof}
	We begin by recalling from Theorem \ref{prop_weakly_pareja} that
	\begin{equation*}
		\lim\limits_{t \to \infty} \mathbb{P}_{(z,x)}\left(Z_t>0\ \Big| \ \underline{\xi}_t> 0\right) = \mathfrak{b}(z,x)>0.
	\end{equation*}
	Thus, appealing to  \eqref{eq_weakly_cota_hirano} we obtain that, 	
	\begin{eqnarray*}
		\mathbb{P}_{(z,x)}\Big(Z_t>0,\  \underline{\xi}_t>0\Big) &=&  \mathbb{P}_{(z,x)}\left(Z_t>0 \ \big|\big. \ \underline{\xi}_t>0\right)  \mathbb{P}_x^{(e)}\left(\underline{\xi}_t>0\right) \\ &\sim& \mathfrak{b}(z,x)\frac{A_\gamma}{ \gamma\kappa^{(\gamma)}(0,\gamma) }e^{\gamma x} \widehat{U}^{(\gamma)}(x) t^{-3/2}e^{\Phi_{\xi}(\gamma)t},
	\end{eqnarray*}
	as $t\to \infty$, which yields the desired result.
\end{proof}

The following lemma tell us that, under the condition that $\ell(\lambda)>C$, for $C>0$,   only  a L\'evy random environment with a high infimum contribute substantially to the non-extinction probability. 

\begin{lemma}\label{lem_cero_weakly}
	Suppose that  $\ell(\lambda)>C$, for $C>0$. Then for  $\delta\in (0,1) $ and  $z, x>0$, we have
	\begin{equation}\label{eq_wekly_cota_cero}
		\lim\limits_{y \to \infty} \limsup_{t\to \infty} t^{3/2} e^{-t\Phi_{\xi}(\gamma)}\mathbb{P}_{(z,x)}\left(Z_t>0,\ \underline{\xi}_{t-\delta}\leq -y\right) = 0.
	\end{equation}
\end{lemma}
\begin{proof}
	The proof of this lemma follows  similar arguments as those used in the proofs of Lemma 6 in Bansaye et al. \cite{bansaye2021extinction} and   Lemma 4.4  in Li et al. \cite{li2018asymptotic}. 
	
	 From \eqref{eq_Laplace}, we deduce the following identity which holds for all $t> 0$,
	\begin{equation}\label{eq_weakly_cota_noext}
		\mathbb{P}_{(z,x)}\big(Z_t > 0 \ \big|\big. \ \xi \big) = 1-\exp\big\{-z v_t(0,\infty, \xi-\xi_0)\big\}.
	\end{equation}
	Similarly as in Lemma 6 in \cite{bansaye2021extinction}, since $\ell(\lambda)>C$ we can bound the functional $v_t(0,\infty,\xi-\xi_0)$ in terms of the exponential functional of the L\'evy process $\xi$, i.e.
	\begin{equation}\label{eq_cota_v}
		v_t(0,\infty,\xi-\xi_0) \leq \Big(\beta C {\tt I}_{0,t}(\beta (\xi-\xi_0))\Big)^{-1/\beta},
	\end{equation}
	where we recall that 
	\begin{equation}\label{eq_func_expon}
		{\tt I}_{s,t}(\beta (\xi-\xi_0)):= \int_{s}^{t} e^{-\beta (\xi_u-\xi_0)} \mathrm{d} u,\qquad \textrm{for} \quad t\ge s\ge 0.
	\end{equation}
	In other words, for   $0< \delta < t$,  we deduce
	\begin{equation}
		\begin{split}\label{eq_cota_debil1}
			\mathbb{P}_{(z,x)}\Big(Z_t>0,\  \underline{\xi}_{t-\delta}\leq -y\Big) &\leq C(z) \mathbb{E}^{(e)}_x\Big[F({\tt I}_{0,t}(\beta (\xi-\xi_0)));\  \underline{\xi}_{t-\delta}\leq -y\Big]  \\ 
			& = C(z) \mathbb{E}^{(e)}\Big[F({\tt I}_{0,t}(\beta \xi));\  \tau^{-}_{-\tilde{y}}\leq t-\delta\Big],
		\end{split}
	\end{equation}
	where $\tilde{y}=y+x$, $\tau^-_{-\tilde{y}} = \inf\{t\geq 0:\xi_t\leq -\tilde{y}\},$  $C(z)=z(\beta C)^{-1/\beta}\lor 1$ and 
	\[
	F(w)=1 - \exp\{-z(\beta C w)^{-1/\beta}\}.
	\]

	To upper bound the right-hand side of \eqref{eq_cota_debil1}, we recall from  Lemma 4.4 in \cite{li2018asymptotic} that  there exists a positive constant $\tilde C$ such that 
	\begin{equation}
\limsup_{t\to \infty} t^{3/2} e^{-t\Phi_{\xi}(\gamma)}  \mathbb{E}^{(e)}\Big[F({\tt I}_{0,t}(\beta \xi));\  \tau^{-}_{-\tilde{y}}\leq t-\delta\Big] \leq \tilde C e^{-\tilde{y}} + \tilde C e^{-(1-\gamma)\tilde{y}}\widehat{U}^{(\gamma)}(\tilde{y}),
	\end{equation}
	which clearly goes to $0$ as $y$ increases, since $\gamma \in (0,1)$ and $ \widehat{U}^{(\gamma)}(y) = \mathcal{O}(y)$ as $y$ goes to $ \infty$. Hence putting all pieces together allow us to deduce our result.
\end{proof}

We are now ready to deduce our second main result. The next result follows the same arguments as those used in the proof of  Theorem 1.2 in \cite{bansaye2021extinction}, we provide its proof for the sake of completeness. 

\begin{proof} [Proof of Theorem \ref{teo_debil}]
	Let $z,x,\epsilon >0$. From Lemma \ref{lem_cero_weakly}, we deduce that we may choose  $y>0$ such that for $t$ sufficiently large 
	\begin{equation}
		\mathbb{P}_{(z,x)}\Big(Z_t>0,\  \underline{\xi}_{t-\delta} \leq -y\Big)   \leq \epsilon \mathbb{P}_{(z,x)}\Big(Z_t>0,\  \underline{\xi}_{t-\delta} > -y\Big).
	\end{equation} 
	Further, since $\{Z_{t} >0\}\subset \{Z_{t-\delta} >0\}$ for $t$ large, we deduce that
	\begin{eqnarray*}
		\mathbb{P}_{z}(Z_t>0) &=&  \mathbb{P}_{(z,x)}\Big(Z_t>0,\  \underline{\xi}_{t-\delta} >-y\Big)  + \mathbb{P}_{(z,x)}\Big(Z_t>0,\  \underline{\xi}_{t-\delta} \leq -y\Big) \\ &\leq & (1+\epsilon) \mathbb{P}_{(z,x+y)}\Big(Z_{t-\delta}>0,\  \underline{\xi}_{t-\delta} > 0\Big).
	\end{eqnarray*}
	In other words, for every $\epsilon >0$ there exists $y'>0$ such that  
	\[\begin{split}
		(1-\epsilon) t^{3/2}e^{-\Phi_{\xi}(1)t}\mathbb{P}_{(z,y')}\Big(Z_t>0,\ & \underline{\xi}_{t} >0\Big) \leq  t^{3/2}e^{-\Phi_{\xi}(1)t}\mathbb{P}_{z}(Z_t>0) \\ &\leq  (1+\epsilon) t^{3/2}e^{-\Phi_{\xi}(1)t}\mathbb{P}_{(z,y')}\Big(Z_{t-\delta}>0,\ \overline{\xi}_{t-\delta} > 0\Big).
	\end{split}\]
	Now, appealing to Lemma \ref{lem_weakly_2}, we  have 
	\begin{equation*}
		\lim\limits_{t \to \infty}t^{3/2}e^{-\Phi_{\xi}(1)t}\mathbb{P}_{(z,y')}\Big(Z_t>0,\ \underline{\xi}_t >0\Big) = \mathfrak{b}(z,y') \frac{A_\gamma}{\gamma\kappa^{(\gamma)}(0,\gamma)}e^{\gamma y'} \widehat{U}^{(\gamma)}(y').
	\end{equation*}
Hence, we obtain 
	\[\begin{split}
		(1-\epsilon)   \frac{A_\gamma}{\gamma\kappa^{(\gamma)}(0,\gamma)} \mathfrak{b}(z,y') e^{\gamma y'} \widehat{U}^{(\gamma)}(y') &\leq \lim\limits_{t \to \infty} t^{3/2}e^{-t\Phi_{\xi}(1)} \mathbb{P}_{z}(Z_t>0)  \\ &\leq (1+\epsilon)   \frac{A_\gamma}{\gamma\kappa^{(\gamma)}(0,\gamma)} \mathfrak{b}(z,y')  e^{\gamma y'} \widehat{U}^{(\gamma)}(y') e^{-\Phi_{\xi}(1)\delta },
	\end{split}\]
	where $y'$ may depend on $\epsilon$ and $z$. Next, we choose  $y'$ in such a way that it goes to infinity as $\epsilon$ goes to 0.  In other words, for any  $y'=y_{\epsilon}(z)$ which goes to $\infty$ as $\epsilon$ goes to 0, we have 
	\[\begin{split}
		0<	(1-\epsilon) \frac{A_\gamma}{\gamma\kappa^{(\gamma)}(0,\gamma)} &\mathfrak{b}(z,y_{\epsilon}(z)) e^{\gamma y_{ \epsilon}(z)}\widehat{U}^{(\gamma)}(y_{\epsilon}(z)) \leq \lim\limits_{t \to \infty} t^{3/2}e^{-\Phi_{\xi}(1)t}\mathbb{P}_{z}(Z_t>0)\\
		 &\leq (1+\epsilon)   \frac{A_\gamma}{\gamma\kappa^{(\gamma)}(0,\gamma)}  \mathfrak{b}(z,y_{\epsilon}(z)) e^{\gamma y_{\epsilon}(z)}\widehat{U}^{(\gamma)}(y_{\epsilon}(z))  e^{-\Phi_{\xi}(1)\delta } < \infty.
	\end{split}
	\]
	Therefore, letting $\epsilon\to 0$, we get
	\[\begin{split}
		0<	\liminf_{\epsilon \to 0}  (1-\epsilon) & \frac{A_\gamma}{\gamma\kappa^{(\gamma)}(0,\gamma)} \mathfrak{b}(z,y_{\epsilon}(z)) e^{\gamma y_{ \epsilon}(z)}\widehat{U}^{(\gamma)}(y_{\epsilon}(z)) \leq\lim\limits_{t \to \infty} t^{3/2}e^{-\Phi_{\xi}(1)t}\mathbb{P}_{z}(Z_t>0)\\  &\leq  \limsup_{\epsilon \to 0}(1+\epsilon)   \frac{A_\gamma}{\gamma\kappa^{(\gamma)}(0,\gamma)} \mathfrak{b}(z,y_{\epsilon}(z)) e^{\gamma y_{\epsilon}(z)}U^{(\gamma)}(y_{ \epsilon}(z)) e^{-\Phi_{\xi}(1)\delta }< \infty.
	\end{split}\]
	Since $\delta$ can be taken arbitrary close to 0,  we deduce 
	\begin{equation*}
		\lim\limits_{t \to \infty} t^{3/2}e^{-\Phi_{\xi}(1)t} \mathbb{P}_{z}(Z_t>0)=\mathfrak{B}(z),
	\end{equation*}
	where
	$$ \mathfrak{B}(z):= \frac{A_\gamma}{\gamma\kappa^{(\gamma)}(0,\gamma)}\lim_{\epsilon \to 0}  \mathfrak{b}(z,y_{\epsilon}(z))e^{\gamma y_{\epsilon}(z)}\widehat{U}^{(\gamma)}(y_{\epsilon}(z)) \in (0,\infty). $$
	Thus the proof is completed.	
\end{proof}

\subsection{The stable case}\label{sec:stable}
Here, we compute the constant $\mathfrak{B}(z)$ in the stable case and verify that it coincides with the constant that appears in Theorem 5.1 in Li and Xu \cite{li2018asymptotic}. To this end, we recall that in the stable case we have  $\psi_0(\lambda)= C\lambda^{1+\beta}$ with $\beta\in (0,1)$ and $C>0$. Moreover, the  backward differential equation \eqref{eq_BDE} can be solved explicitly (see e.g. Section 5 in \cite{he2018continuous}), that is for any $\lambda\geq 0$ and $s\in [0,t]$,
\begin{equation}\label{eq:vstable}
	v_t(s,\lambda, \xi) = \Big(\lambda^{-\beta} + \beta C\texttt{I}_{s,t}(\beta \xi)\Big)^{-1/\beta},
\end{equation}
where $\texttt{I}_{s,t}(\beta \xi)$  denotes the exponential functional of  the L\'evy process $\beta \xi$ defined in \eqref{eq_expfuncLevy}.

Next, we observe that, for any $z,x >0$, the constant $\mathfrak{b} (z,x)$ defined in Theorem \ref{prop_weakly_pareja} can be rewritten as follows
\begin{equation*}
	\mathfrak{b}(z,x) = 1-  \lim\limits_{\lambda \to \infty} \lim\limits_{s\to \infty}  \gamma\kappa^{(\gamma)}(0,\gamma)  \int_{0}^{\infty} e^{-\gamma y} U^{(\gamma)} (y)  R_{s,\lambda}(z,x,y){\rm d} y,
\end{equation*}
where
\[ R_{s,\lambda}(z,x,y) :=  \int_{0}^{1} \int_{0}^{\infty} w^{\textbf{u}} \mathbb{P}^{(\gamma),\uparrow}_{(z,x)}\big(\mathcal{U}_s \in \mathrm{d} \textbf{u}\big) \mathbb{P}^{(e,\gamma),\downarrow}_{-y}\big(\widehat{W}_s(\lambda) \in \mathrm{d} w\big).\]
In order to find an explicit expression of the previous double integral we use Proposition 3.3. in \cite{bansaye2021extinction} which claims that for any $z,x>0$ and $\theta \geq 0$, we have
$$\mathbb{E}_{(z,x)}^{(\gamma), \uparrow} \Big[\exp\left\{-\theta Z_s e^{-\xi_s}\right\} \Big] = \mathbb{E}^{(e, \gamma), \uparrow}_x\Big[\exp\left\{-z v_s(0, \theta e^{-x}, \xi-x)\right\}\Big]. $$
It follows that 
\begin{equation*}
	\begin{split}
		R_{s,\lambda}(z,x,y)&=  \int_{0}^{1} \mathbb{E}_{(z,x)}^{(\gamma), \uparrow}\left[w^{\mathcal{U}_s}\right] \mathbb{P}^{(e,\gamma),\downarrow}_{-y}\big(\widehat{W}_s(\lambda) \in \mathrm{d} w\big) \\ 
		& = \int_{0}^{1} \mathbb{E}_{(z,x)}^{(\gamma), \uparrow}\Big[\exp\left\{\log(w)Z_se^{-\xi_s}\right\}\Big] \mathbb{P}^{(e,\gamma),\downarrow}_{-y}\big(\widehat{W}_s(\lambda) \in \mathrm{d} w\big)  \\ 
		& = \int_{0}^{1} \mathbb{E}_{x}^{(e, \gamma), \uparrow}\Big[\exp\left\{-z v_s(0,- \log(w)e^{-x}, \xi-x)\right\}\Big] \mathbb{P}^{(e,\gamma),\downarrow}_{-y}\big(\widehat{W}_s(\lambda) \in \mathrm{d} w\big)\\ 
		& =  \int_{0}^{\infty} \int_{0}^{\infty} e^{-ze^{-x}(\beta C w +\beta C u)^{-1/\beta}}  \mathbb{P}^{(\gamma),\uparrow}_{(z,x)}\big(\texttt{I}_{0,\infty}(\beta \xi) \in \mathrm{d} w\big) \mathbb{P}^{(e,\gamma),\downarrow}_{-y}\Big(\texttt{I}_{0,\infty}(\beta \widehat{\xi}) \in \mathrm{d} u\Big),
		% \textcolor{purple}{\mathbb{Q}_{(y,x)}^{(e, \gamma)}\left[\exp\left\{-ze^{-x}\left(\beta C 	\texttt{I}_{0,s}(\beta \xi) + \lambda^{-\beta} + \beta C 	\texttt{I}_{0,s}(\beta \widehat{\xi})\right)^{-1/\beta}\right\}\right],}
	\end{split}
\end{equation*}
where in the last equality we have used \eqref{eq:vstable}. Thus putting all pieces together and appealing to the Dominated Convergence Theorem, we deduce 
\begin{equation*}
	\begin{split}
		\mathfrak{b} (z,x)& = 1- \gamma\kappa^{(\gamma)}(0,\gamma)  \int_{0}^{\infty} e^{-\gamma y} U^{(\gamma)} (y)   \lim\limits_{\lambda \to \infty} \lim\limits_{s\to \infty} R_{s,\lambda}(z,x,y){\rm d} y \\ & =  \gamma\kappa^{(\gamma)}(0,\gamma)  \int_{0}^{\infty} e^{-\gamma y} U^{(\gamma)} (y) G_{z,x}(y){\rm d} y,
	\end{split}
\end{equation*}
where $G_{z,x}(\cdot)$ is as  in \eqref{functG}. Therefore, we have that the limiting constant in the stable case is given by
\begin{equation*}
	\begin{split}
		\mathfrak{B}(z)&:=\frac{A_\gamma}{\gamma\kappa^{(\gamma)}(0,\gamma)} \lim_{x\to \infty} \mathfrak{b}(z,x)e^{\gamma x}\widehat{U}^{(\gamma)}(x)\\ & =    A_\gamma \lim\limits_{x\to \infty} e^{\gamma x} \widehat{U}^{(\gamma)}(x)\int_{0}^{\infty} e^{-\gamma y} U^{(\gamma)}(y)G_{z,x}(y){\rm d} y,
	\end{split}
\end{equation*}
as expected.

~\\

{\bf Acknowledgements:} N.C.-T. acknowledges support from CONACyT-MEXICO grant no. 636133 and financial support from the University of G\"ottingen. This work was concluded whilst N.C.-T. was visiting CIMAT whom she also acknowledges for their hospitality.

%\textbf{Acknowledgements}:  
%N.C.-T. acknowledges support from CONACyT-MEXICO grant no. 636133.

\bibliographystyle{abbrv}
\bibliography{references}

\end{document}